\documentclass[12pt,a4paper,reqno]{amsart}

\usepackage[a4paper,inner=1in,outer=1.26in,top=1in,bottom=0.8in]{geometry}
\usepackage{amsmath,amssymb,amsfonts}
\usepackage{graphicx,hyperref}
\usepackage{tikz}
\usetikzlibrary{calc,positioning}
\usepackage{xcolor}
\definecolor{eyebg}{RGB}{15,226,39} % 浅绿色背景\pagecolor{eyebg}
\newcommand{\lightzero}{{\color{black!25}0}}
\numberwithin{equation}{section}

\newtheorem{theorem}{Theorem}[section]
\newtheorem{proposition}[theorem]{Proposition}
\newtheorem{lemma}[theorem]{Lemma}
\newtheorem{corollary}[theorem]{Corollary}
\theoremstyle{definition}
\newtheorem{definition}[theorem]{Definition}

\theoremstyle{remark}
\newtheorem*{remark}{Remark}

\newcommand{\R}{\mathbb R}

\title[The Ces\`aro Operator, Rhaly operators, and the Toeplitz Algebra]
      {The Ces\`aro Operator, Rhaly operators, and the Toeplitz Algebra}

\author{Yuanqi Sang}
\address{School of Mathematics, Southwestern University of Finance and Economics,
  Chengdu, 611130, China}
\email{sangyq@swufe.edu.cn}
\date{}

\begin{document}

\begin{abstract}
Barr\'{\i}a and Halmos [\emph{Trans. Amer. Math. Soc.} (1982), (10) Example] asked whether the Ces\`aro operator belongs to the Toeplitz algebra. We answer this question affirmatively and prove that the Ces\`aro operator is the product of a Toeplitz operator and a continuous function of a self-adjoint Toeplitz operator.
Rhaly operators are a natural generalization of the Ces\`aro operator. We give necessary and sufficient conditions for a Rhaly operator to belong to the Toeplitz algebra or the Hankel algebra. 
The main approach is to use the Fejér kernels to get localized
estimates for operators. This yields a necessary condition for an operator to belong to the Hankel algebra.
As an application, we find a class of Rhaly operators that are asymptotic Toeplitz operators and essentially commute with the unilateral shift, but are not in the Hankel algebra. This answers (17) Question in the same paper.
\end{abstract}

\maketitle

\section{Introduction}

Let $\mathbb D$ and $\mathbb T$ denote the open unit disk and the unit circle in $\mathbb C$, respectively.
The Hardy space $H^{2}(\mathbb D)$
consists of all analytic functions $f$ on $\mathbb D$ such that
\begin{align*}
\sup_{0<r<1} \left( \frac{1}{2\pi}\int_{0}^{2\pi} |f(re^{it})|^2\,dt \right)^{1/2} <\infty.
\end{align*}
As usual, we write
$L^2(\mathbb T)=L^2(\mathbb T,\frac{d\theta}{2\pi})$, and identify $H^2(\mathbb D)$ with
$H^2(\mathbb T)$, the closed subspace of $L^2(\mathbb T)$ consisting
of functions whose negative Fourier coefficients vanish.

The \emph{Ces\`aro operator} $\mathcal C$ on \(H^2(\mathbb D)\) is defined by
\begin{align*}
(\mathcal Cf)(z)=\frac1z\int_{0}^{z}\frac{f(w)}{1-w}\,dw, \qquad f\in H^{2}(\mathbb D).
\end{align*}
For a comprehensive treatment, see \cite{MashreghiRoss2026}.
For a complex number sequence $a=(a_n)_{n\geq 1},$ the Rhaly matrices $R_a$ \cite{Rhaly1989} defined as 
\begingroup
\renewcommand{\arraystretch}{1}
\setlength{\arraycolsep}{6pt}
\begin{align*}
\begin{pmatrix}
a_1 & \lightzero & \lightzero & \lightzero & \cdots \\
a_2 & a_2 & \lightzero & \lightzero & \cdots \\
a_3 & a_3 & a_3 & \lightzero & \cdots \\
a_4 & a_4 & a_4 & a_4 & \cdots \\
\vdots & \vdots & \vdots & \vdots & \ddots
\end{pmatrix}.
\end{align*}
\endgroup
If $a_n=\frac{1}{n},$ the Rhaly matrix $R_a$ is the Ces\`aro operator.
Note that
$R_a$ is bounded on $\ell^2$ if and only if
$\sup_{N\geq 1}N\sum_{n\geq N}|a_n|^2<\infty$
(see \cite[Theorem~2]{Bennett1987} or
\cite[Theorem~2.3]{BDS2025}).
If $R_a$ is bounded on $\ell^2,$ the corresponding Rhaly operator $\mathcal R_a$ on $H^2(\mathbb T)$ is defined by
\[(\mathcal R_af)(z) =\sum_{n=1}^{\infty} \left(a_n\sum_{k=1}^{n}f_k\right)z^{n-1},\qquad f(z)=\sum_{k=1}^{\infty}f_kz^{k-1}\in H^{2}(\mathbb T).\]

Let $P$ denote the orthogonal projection from $L^{2}(\mathbb{T})$ onto
$H^{2}(\mathbb{T})$. Let $L^{\infty}(\mathbb{T})$ denote the space of
essentially bounded measurable functions on $\mathbb{T}$. 
For
$\phi\in L^{\infty}(\mathbb{T})$ and $x\in H^{2}(\mathbb T),$
the Toeplitz operator $T_{\phi}$ and the Hankel operator $H_{\phi}$ are defined by
\begin{align*}
T_{\phi} x &= P(\phi x),\\
H_{\phi} x &= P\bigl(J\phi x\bigr),
\end{align*}
where 
$J$ is the flip operator on $L^2(\mathbb T)$, given by
$
(Jg)(z)=\bar z g(\bar z).
$
In particular, $S=T_{z}$ is the unilateral shift.

We list several operator classes discussed in this article.
\begin{itemize}
\item $\mathbf T_0$ denotes the algebra consisting of all finite sums of finite
products of Toeplitz operators, without taking the norm closure.
\item $\mathbf T_c$ denotes the $C^*$-algebra generated by
$\{T_{\phi}:\phi \ \text{is continuous functions on}\ \mathbb{T}\}$.
\item $\mathbf T$ denotes the $C^*$-algebra generated by
$\{T_{\phi}:\phi\in L^{\infty}(\mathbb T)\}$.
\item $\mathbf T^{\infty}$ denotes the set of all bounded operators $T$
such that $\{S^{*k}TS^k\}$ is strongly convergent.
\item $\mathbf T^{+}$ denotes the $C^*$-algebra generated by all bounded Toeplitz operators and Hankel operators. We call $\mathbf T^{+}$ the Hankel algebra.
\item $\mathbf Q$ denotes the commutator ideal of $\mathbf T.$
\item $\mathbf E$ denotes the essential commutant of $S$.
\end{itemize}
The following inclusions hold:
\begin{align*}
\mathbf T_c\subset \mathbf T\subset \mathbf T^{+}\subset \mathbf T^{\infty},
\quad
\mathbf T_0\subset \mathbf T\subset \mathbf E.
\end{align*}

\begin{figure}[htbp]
\centering
\begin{tikzpicture}[
    x=0.9cm, y=0.9cm,
    line join=round, line cap=round,
    % ---- styles (edit colours here) ----
    frame/.style   = {line width=0.9pt, draw=black!80},
    cut/.style     = {line width=0.6pt, draw=Tcol!85},
    lab/.style     = {font=\normalsize, inner sep=1.5pt},
    sublab/.style  = {font=\footnotesize, text=black!60},
  ]
  \definecolor{Tcol}{RGB}{38,84,140}    % slate blue  : Toeplitz family
  \definecolor{Ecol}{RGB}{168,62,50}    % brick red   : E

  % ---- key points ----
  \coordinate (O)  at (0,0);        % bottom-left
  \coordinate (B)  at (10,0);       % bottom-right
  \coordinate (EL) at (0,5.7);      % apex of E
  \coordinate (TR) at (10,5.7);     % apex of T^infty

  % points on the hypotenuse O--TR where the cuts start
  \coordinate (tc) at ($(O)!0.352!(TR)$);
  \coordinate (tt) at ($(O)!0.500!(TR)$);   % = crossing of the two hypotenuses
  \coordinate (tp) at ($(O)!0.645!(TR)$);
  % feet of the cuts on the base
  \coordinate (bc) at (4.00,0);
  \coordinate (bt) at (5.79,0);
  \coordinate (bp) at (7.62,0);

  % ---- nested shading: T_c < T < T^+ < T^infty ----
  \fill[Tcol!5]  (O) -- (B) -- (TR) -- cycle;
  \fill[Tcol!10] (O) -- (bp) -- (tp) -- cycle;
  \fill[Tcol!16] (O) -- (bt) -- (tt) -- cycle;
  \fill[Tcol!23] (O) -- (bc) -- (tc) -- cycle;

  % ---- cuts ----
  \draw[cut] (tc) -- (bc);
  \draw[cut] (tt) -- (bt);
  \draw[cut] (tp) -- (bp);

  % ---- outlines ----
  \draw[frame, draw=Ecol] (B) -- (EL) -- (O);
  \draw[frame, draw=Tcol] (O) -- (TR) -- (B);
  \draw[frame]           (O) -- (B);

  % ---- the point C ----
  \fill[black] (5.26,1.15) circle (1.8pt);
  \node[lab, left=1pt] at (5.26,1.05) {$\mathcal C$};

  % ---- labels ----
  \node[lab, text=Ecol, anchor=north west]
    at (0.15,5.26) {$\mathbf E$};
  \node[lab, text=Tcol, anchor=north east]
    at (9.95,5.26) {$\mathbf T^\infty$};
  \node[lab] at (6.30,3.25) {$\mathbf T^+$};
  \node[lab] at (4.70,2.30) {$\mathbf T$};
  \node[lab] at (3.20,1.40) {$\mathbf T_c$};
  \node[sublab] at (6.45,0.37) {(e)};
  \node[sublab] at (8.35,0.37) {(f)};
\end{tikzpicture}
\caption{}
\label{fig:toeplitz-classes}
\end{figure}

Barría and Halmos \cite{BarriaHalmos1982} studied the relationships  among several classes of operators associated with Toeplitz operators. In 
\cite[(10) Example]{BarriaHalmos1982},
they proved that $\mathcal C\in(\mathbf E\cap \mathbf T^{\infty})\setminus\mathbf T_0.$  They also showed that
$\mathcal C$ is not compact and its
asymptotic symbol is zero. Since the kernel of the symbol mapping on  $\mathbf T_c$ is the ideal of compact operators
\cite{Coburn1967}, $\mathcal C\notin\mathbf T_c$.
Then Barría and Halmos pose the following question.

\textbf{Question 1.}
\begin{equation*}
\emph{Is $\mathcal C$ in $\mathbf T$?}
\label{ques:1}
\end{equation*}

Engli\v{s} \cite{Englis1995} introduced the \(C^*\)-algebra
$
\mathcal A_1=\{T_\varphi+H:\varphi\in L^\infty(\mathbb T),\,H\in\mathcal J\},
$
where
$\mathcal J=\{H\in B(H^{2}(\mathbb T)):\|Hk_\lambda\| \ \text{and} \ \|H^{*}k_\lambda\|\to 0\ \text{radially}\}$
and $k_\lambda(z)=\sqrt{1-|\lambda|^2}/{(1-\overline\lambda z)}.$
He proved that 
$\mathbf T\subset \mathcal A_1, \mathbf T \cap \mathcal J=\mathbf Q$
and 
$\mathcal{C}\in \mathcal J\subset\mathcal A_1$.
Moreover, 
Mart{\'{\i}}nez-Avenda{\~n}o \cite{MartinezAvendano2002} showed that
$\mathcal C$ is essentially Hankel, but
$\mathcal C$ is not a compact perturbation of any bounded
Hankel operator.
Let $\mathbf E(\mathbf T)$ denote the essential commutant
of $\mathbf T$ and $
(V_gf)(z)=\int_0^z f(w)g'(w)dw.
$
Bellavita and Stylogiannis \cite{BellavitaStylogiannis2024} showed that if 
$g\in\mathrm{BMOA}\setminus\mathrm{VMOA},$ then
$V_g\notin\mathbf E(\mathbf T).$ 
Taking $g=-\log(1-z)$ gives
$V_g=S\mathcal C$ and 
$\mathcal C\notin\mathbf E(\mathbf T)$. Davidson
\cite{Davidson1977} characterized
$\mathbf E(\mathbf T)$ as the $C^*$-algebra generated by
the Toeplitz operators with quasicontinuous symbols.

Known results show that the Cesàro operator does not belong to certain subalgebras of the Toeplitz algebra $\mathbf T$, such as $\mathbf E(\mathbf T)\supset \mathbf T_c$. However, it does belong to some larger algebras which contain the Toeplitz algebra $\mathbf T,$ such as $\mathbf T^{\infty}\cap \mathbf E$ and $\mathcal A_1.$
A 2026 monograph on the Cesàro operator \cite{MashreghiRoss2026} states that  \textbf{Question 1}
remains open.

To better understand \textbf{Question 1}, we propose the following more general question.

\textbf{Question 2.}
\begin{equation*}
\text{\emph{Which Rhaly matrices belong to $\mathbf T$?}}
\label{ques:2}
\end{equation*}

In  \cite[(16) Question]{BarriaHalmos1982}, Barría and Halmos asked whether  there  exists an operator in
$(\mathbf E\cap \mathbf T^{+})\setminus\mathbf T$? 
In other words, is region $(e)$ in Figure \ref{fig:toeplitz-classes} empty?
Chen and Chen \cite{ChenChen1990}, Barría \cite{Barria1996},
and Chen, Guo, Izuchi and Zheng \cite{ChenGuoIzuchiZheng2005} answer this question by constructing  Hankel operators that belong to $\mathbf E\setminus \mathbf T$.
Moreover, in \cite[(17) Question]{BarriaHalmos1982}, Barría and Halmos pose the following question.
Namely, is the region $(f)$ in Figure \ref{fig:toeplitz-classes} empty?

\textbf{Question 3.}
\begin{equation*}
\emph{Is there an operator in $(\mathbf E\cap\mathbf T^\infty)$ that is not in $\mathbf T^+$ ?}
\label{ques:3}
\end{equation*}

The aim of this paper is to answer the three questions and to develop new methods for determining whether an operator belongs to a Toeplitz algebra $\mathbf T$.
Douglas \cite[Theorem~7.11]{Douglas1998} established  the following  structure theorem for Toeplitz algebra $\mathbf T.$
\begin{align}\label{maintool}
\mathbf T=\{T_{f}+A:f\in L^{\infty}(\mathbb T),\ A\in\mathbf Q\}.
\end{align}
However, when determining whether some operators belong to a Toeplitz algebra, this theorem provides only limited information.
For example, Barría and Halmos
\cite[(19) Question]{BarriaHalmos1982} asked which projections belong
to $\mathbf T$. 
Limited progress has been made on this problem; see
\cite{DanDingGuoSang2020}.

Throughout the paper,  $\mathbb R$ denotes the set of all real numbers, for any measurable subset $E$ of either
$\mathbb R$ or $\mathbb T$, we denote by $\chi_E$ its characteristic
function, defined by
\[
\chi_E(t)
=
\begin{cases}
1, & t\in E,\\
0, & t\notin E.
\end{cases}
\]
We now state the main
results of this paper.
\begin{theorem}\label{thm:main1}
The Cesàro operator $\mathcal C$ admits the following factorization:
\begin{align*}
\mathcal C=(I+S^{*})\varphi (T_{\mathord{\scalebox{1}{$\chi$}}_{\mathbb T_-}}),
\end{align*}
where 
\[ 
\varphi(s)=
\left\{
\begin{array}{c@{\quad}l}
\displaystyle
\left[1-\frac{i}{\pi}\log\left(\frac{1}{s}-1\right)\right]^{-1},
& 0<s<1,\\[1.2ex]
0,
& s=0\ \text{or}\ s=1,
\end{array}
\right.\] and $\mathbb T_{-}=\{e^{i\theta}:\pi<\theta<2\pi \}.$
Hence,
$\mathcal C\in \mathbf Q\subset\mathbf T$.
\end{theorem}
\begin{theorem}\label{thm:main2}
If the Rhaly operator $\mathcal R_a$ is bounded on $H^{2}(\mathbb{T}),$ then the following are equivalent.
\begin{enumerate}
  \item $\mathcal R_{a}\in \mathbf T.$
  \item $\mathcal R_{a}\in \mathbf T^{+}.$
  \item The sequence $a=(a_n)_{n\geq 1}$ satisfies
\[\lim_{\delta\to 0}
\sup_{\substack{N\geq1\\0\leq k\leq \delta N}}
\left(N\sum_{n=N}^\infty|a_{n+k}-a_n|^2\right)=0.\]
\end{enumerate}
Moreover, 
the following are equivalent.
\begin{enumerate}
  \item[$(a)$] $\mathcal R_a\in \mathbf T_c.$
  \item[$(b)$] $\mathcal R_a$ is compact.
  \item[$(c)$]  $\lim\limits_{N\to\infty}N\sum_{n=N}^{\infty}|a_n|^2=0.$
\end{enumerate}
\end{theorem}
\begin{theorem}\label{thm:main3}
There exist nonzero bounded Rhaly operators that are in
$(\mathbf E\cap \mathbf T^{\infty})\setminus\mathbf T^{+}$.
\end{theorem}
Recently, Bellavita and Stylogiannis \cite{BellavitaStylogiannis2026DiscreteMellin} proved that sampled-ratio matrices belong to $\mathbf Q$.
The Ces\`aro operator is a member of this family.
Sampled-ratio matrices and Rhaly operators form two distinct classes of operators, and their intersection consists precisely of the constant multiples of the Ces\`aro operator. The method in this paper is independent of \cite{BellavitaStylogiannis2026DiscreteMellin}.

The paper is organized as follows. In Section 2, using the unitary 
equivalence between Toeplitz operators and Wiener-Hopf operators,
we prove that Ces\`aro operator on $L^2(\mathbb R_{+})$ is a product of a Wiener-Hopf operator and 
the adjoint of the Hardy operator. Using the Mellin transform,
we show that the Hardy operator is a continuous functional calculus
of a Wiener-Hopf operator. Thus, the Ces\`aro operator belongs to the Toeplitz algebra.
In Section 3, using an isometry, we transfer the weighted
Hardy operator to $H^{2}(\mathbb T)$ and show that it belongs to $\mathbf Q.$ In Section 4, we obtain a necessary condition for an operator to belong to the Hankel algebra.
This condition plays a key role in this paper.
In Section 5, we introduce two sequence spaces to characterize
the conditions of Rhaly operators in $\mathbf{E}$, $\mathbf{T}$, $\mathbf{T}^{+}$ and $\mathbf{T}^{\infty}.$ 
\section{Toeplitz and Wiener--Hopf Operators}
In this section, we recall the classical unitary equivalence between 
Toeplitz operators and Wiener--Hopf operators; 
see, for example, \cite{Devinatz1967}, \cite[4.2]{Nikolski2020} and \cite[Part A, Chapter 6; Part B, 4.6]{Nikolski2002}.
We mainly follow the notation used in Nikolski’s books \cite{Nikolski2002,Nikolski2020}.

Throughout, we write
$\mathbb C_+=\{z\in\mathbb C:\operatorname{Im}z>0\},\mathbb R_+=(0,+\infty)$ and 
$\mathbb R_-=(-\infty,0).$
Let $\omega:\mathbb C_+\longrightarrow\mathbb D$ be the conformal map given by
$\omega(\eta)=\frac{\eta-i}{\eta+i}\ (\eta \in \mathbb C_+)$
and
$
\omega^{-1}(z)=i\frac{1+z}{1-z}\ (z\in \mathbb D).
$
The induced boundary map $\omega:\mathbb R \to \mathbb T\setminus\{1\}$ is a bijection.
Define
\begin{equation*}
 (Uf)(x)=\frac{1}{\sqrt\pi\,(x+i)}
 f\!\left(\frac{x-i}{x+i}\right),
 \qquad x\in\mathbb R.
\end{equation*}
Then \(U:L^2(\mathbb T)\to L^2(\mathbb R)\) is unitary. 
We define $H^2(\mathbb R)$ to be the closed subspace of $L^2(\mathbb R)$ given by
$UH^2(\mathbb T).$
We identify
\begin{align*}
H^2(\mathbb C_+) = \left\{ f\in\operatorname{Hol}(\mathbb C_+): \sup_{y>0} \int_{\mathbb R}|f(x+iy)|^2\,dx<\infty \right\}
\end{align*}
with $H^2(\mathbb R)$ via boundary values.

Let $\mathcal F:L^2(\mathbb R)\to L^2(\mathbb R)$ be the unitary
Fourier transform given by
\begin{equation*}
 (\mathcal F f)(t)=\frac{1}{\sqrt{2\pi}}\int_\mathbb{R} f(x)e^{-ixt}dx,
 \qquad
 (\mathcal F^{-1}g)(x)=\frac{1}{\sqrt{2\pi}}
 \int_\R g(t)e^{ixt}dt.
\end{equation*}
We identify $L^2(\mathbb R_+)$ with $\chi_{\mathbb R_+}L^2(\mathbb R)$. The Paley--Wiener theorem \cite{PaleyWiener1934} implies 
$\mathcal F H^2(\mathbb C_+)=L^2(\mathbb R_+).$
Thus,
\begin{equation*}
 \mathcal U=\mathcal F U:
 H^2(\mathbb D)\longrightarrow L^2(\mathbb R_+)
\end{equation*}
is unitary.
Let \(\mathcal S(\mathbb R)\) be the Schwartz space and let
\(\mathcal S'(\mathbb R)\) be the space of tempered distributions. 
The set of all pseudo-measures is denoted by
\begin{equation*}
 \mathcal{PM}(\mathbb R)
 =\{k\in\mathcal S'(\mathbb R):\mathcal F k\in L^\infty(\mathbb R)\}.
\end{equation*}
For \(k\in\mathcal{PM}(\mathbb{R})\), the Wiener--Hopf operator with
kernel \(k\) is the bounded operator
\(W_k:L^2(\mathbb R_+)\to L^2(\mathbb R_+)\) initially defined by
\begin{equation*}
 W_kf=P^+(k*f),
 \quad f\in\{f\in\mathcal S(\mathbb{R}):\operatorname{supp}f\subset\mathbb R_+\}=\mathcal S(\mathbb R_{+}),
\end{equation*}
where $P^+g=\chi_{\mathbb R_+}g$
is the orthogonal projection from \(L^2(\mathbb{R})\) onto
$L^2(\mathbb R_+)$.
The function
$\Phi=\sqrt{2\pi} \mathcal F^{-1}k \in L^\infty(\mathbb{R})$
is called the symbol of \(W_k\). Let
\begin{align}
\hat{\Phi}=\frac{1}{\sqrt{2\pi}} \mathcal F\Phi.
\end{align}
The following unitary equivalence between Toeplitz operators and Wiener--Hopf 
operators will be useful in this paper.
\begin{proposition}\cite[Corollary 4.2.3]{Nikolski2020}\label{pr:wiener}
Let $k \in \mathcal{PM}(\mathbb{R}).$ A Wiener--Hopf operator
\begin{align*}
W_k: L^2\left(\mathbb{R}_{+}\right) \rightarrow L^2\left(\mathbb{R}_{+}\right)
\ \text{with symbol} \ \ \Phi=\sqrt{2\pi}\mathcal{F}^{-1} k
\end{align*}
is unitarily equivalent to a Toeplitz operator 
$T_{\varphi}: H^2(\mathbb{D}) \rightarrow H^2(\mathbb{D})$,
\begin{align*}
W_k=\mathcal{U} T_{\varphi} \mathcal{U}^{-1}, \quad 
\text{where} \quad \varphi=\Phi \circ \omega^{-1}.
\end{align*}  
\end{proposition}
\begin{proposition}\label{WFtau}\cite[Theorem 4.2.2]{Nikolski2020}
Let $W$ be a bounded operator on $L^2(\mathbb R_{+}).$ 
The following assertions are equivalent.
\begin{enumerate}
  \item $W=W_k,k\in  \mathcal{PM}(\mathbb R).$
  \item $\mathcal F^{-1}W\mathcal F$ is a Toeplitz operator $T_{\Phi}$ on $H^{2}(\mathbb C_{+}).$
  \item For every $t>0,\tau^{*}_{t}W\tau_{t}=W.$
\end{enumerate}
If $(1)$-$(3)$ hold, then
$\Phi=\sqrt{2\pi}\mathcal{F}^{-1} k.$
Here \[(\tau_{t}f)(x)
=\begin{cases}
 0,&0<x\leq t,\\
 f(x-t),&x>t.
 \end{cases}\]
\end{proposition}
Brown, Halmos, and Shields \cite{BrownHalmosShields1965} initially
studied the following three operators:
\begin{align*}
(\mathcal Ca)(n)
&=
\frac{1}{n}\sum_{j=1}^{n}a_j,
\qquad\qquad a=(a_n)_{n\geq 1}\in \ell^2,\\
(\mathcal C_1f)(x)
&=
\frac{1}{x}\int_0^x f(y)\,dy,
\qquad 0<x<1, \quad f\in L^2(0,1),
\end{align*}
and
\begin{align}\label{c-infty}
(\mathcal C_\infty f)(x)
=
\frac{1}{x}\int_0^x f(y)\,dy,
\qquad 0<x<\infty,\quad f\in L^2(\mathbb R_{+}).
\end{align}
They determined the norms and spectra of these operators and proved that
$I-\mathcal C_1^*$ and $I-\mathcal C_\infty^*$ are unitarily equivalent
to the unilateral and bilateral shifts of multiplicity one, respectively.
The second equivalence is closely related to our analysis of
$\mathcal C_\infty^*$ below.
These results laid the foundation for much of the subsequent work.
Further studies have used reproducing kernels, 
semigroups of composition operators,
multiplication operator models, and the Laplace, Mellin, and Fourier
transforms to study its cyclicity and invariant subspaces; see, for example,
\cite{KrieteTrutt1971,GallardoGutierrezPartington2024,
GallardoGutierrezPartingtonRoss2025Hardy,
GallardoGutierrezPartingtonRoss2026,BelliGulRossSiskakis2026}.
These tools enable us to solve 
\textbf{Question 1.}

In this section, we use the unitary operator $\mathcal U$ to 
get the representation of the Ces\`aro operator $\mathcal C$ on $L^2(\mathbb R_+)$.

\begin{lemma}\label{le:upperhalfplane}
If \(F\in H^{2}(\mathbb C_{+})\), then
\begin{align*}
(U\mathcal C U^{-1}F)(\zeta)=\frac{1}{\zeta-i}
  \int_i^\zeta F(\eta) \,d\eta,
  \qquad
 \zeta\in\mathbb C_+\setminus\{i\}.
\end{align*}
\end{lemma}

\begin{proof}
Let \(f=U^{-1}F\in H^{2}(\mathbb D)\), so $F=Uf$ and
\begin{align}\label{eq:Ff}
F(\eta)=(Uf)(\eta)=\frac{f(\omega(\eta))}{\sqrt{\pi}(\eta+i)},
\qquad \eta\in\mathbb C_+.
\end{align}
For \(\zeta\in\mathbb C_+\setminus\{i\}\), we have
\begin{equation}\label{eq:JCf}
\begin{aligned}
 \bigl(U \mathcal C f\bigr)(\zeta)
 &=\frac{1}{\sqrt{\pi}(\zeta+i)}
   \bigl(\mathcal C f\bigr)\bigl(\omega(\zeta)\bigr)\\
 &=\frac{1}{\sqrt{\pi}(\zeta+i)}
   \frac{1}{\omega(\zeta)}
   \int_0^{\omega(\zeta)}
        \frac{f(z)}{1-z}\,dz\\
 &=\frac{1}{\sqrt{\pi}(\zeta-i)}
  \int_0^{\omega(\zeta)}
       \frac{f(z)}{1-z}\,dz,
 \qquad \text{since}\quad
 \omega(\zeta)=\frac{\zeta-i}{\zeta+i}.
\end{aligned}
\end{equation}
We now make the change of variables
$z=\omega(\eta)=\frac{\eta-i}{\eta+i}.$
Since
\[
\omega'(\eta)
 =\frac{2i}{(\eta+i)^2}
 \qquad\text{and}\qquad
1-\frac{\eta-i}{\eta+i}
   =\frac{2i}{\eta+i},
\]
it follows that
\begin{align*}
 \int_0^{\omega(\zeta)}
       \frac{f(z)}{1-z}\,dz
 &=\int_i^\zeta
   \frac{f(\omega(\eta))}
        {1-\omega(\eta)}
   \omega'(\eta)\,d\eta\\
 &=\int_i^\zeta
   \frac{f(\omega(\eta))}{\eta+i}\,d\eta\\
 &=\int_i^\zeta
   \frac{\sqrt{\pi}(\eta+i)F(\eta)}
        {\eta+i}\,d\eta
        \qquad \text{by \eqref{eq:Ff}}\\
 &=\sqrt{\pi}\int_i^\zeta F(\eta)\,d\eta.
\end{align*}
Substituting this identity into \eqref{eq:JCf} and recalling that
\(f=U^{-1}F\), we obtain
\[
(U\mathcal C U^{-1}F)(\zeta)
=\frac{1}{\zeta-i}\int_i^\zeta F(\eta) d\eta.
\]
\end{proof}
We denote by $C_c^\infty(\mathbb R_+)$ the space of all infinitely
differentiable functions with compact support in $\mathbb R_+$. Recall that
$C_c^\infty(\mathbb R_+)\subsetneq \mathcal S(\mathbb R_+),$
and that $C_c^\infty(\mathbb R_+)$ is dense in $L^2(\mathbb R_+)$.
\begin{lemma}\label{2u}
Let \(k_{R}(y)=e^{y}\chi_{\mathbb R_{-}}(y)\). Then
\begin{align}\label{fenjie}
\mathcal U \mathcal C \mathcal U^{-1}=(I-W_{k_R})\mathcal C^*_\infty.
\end{align}
\end{lemma}

\begin{proof}
Suppose that $f\in C_c^\infty(\mathbb R_+).$
Set \(F=\mathcal F^{-1}f\in H^{2}(\mathbb C_+)\). Thus,
\begin{align*}
 F(\eta)=\frac{1}{\sqrt{2\pi}}
 \int_0^\infty e^{i\eta x}f(x)\,dx,
 \qquad \eta\in \mathbb C_+.
\end{align*}
By Lemma~\ref{le:upperhalfplane}, we have
$
(U\mathcal C U^{-1}F)(\zeta)
=\frac{1}{\zeta-i}\int_i^\zeta F(\eta) d\eta.$
For \(\zeta\in\mathbb C_+\setminus\{i\}\), Fubini's theorem gives
\begin{equation*}
\begin{aligned} 
\frac{1}{\zeta-i}\int_i^\zeta F(\eta)d\eta 
&=\frac{1}{\sqrt{2\pi}}\int_0^\infty\frac{f(x)}{\zeta-i}\left(\int_i^{\zeta}e^{i\eta x}d\eta\right)dx\\
&=\frac{1}{\sqrt{2\pi}}\int_0^\infty f(x)\frac{e^{i\zeta x}-e^{-x}} {ix(\zeta-i)}\,dx\\ 
&=\frac{1}{\sqrt{2\pi}}\int_0^\infty \frac{f(x)}{x} \left( \int_0^x e^{i\zeta t}e^{-(x-t)}\,dt \right)dx\\ 
&=\frac{1}{\sqrt{2\pi}}\int_0^\infty \left( \int_t^\infty e^{-(x-t)} \frac{f(x)}{x}\,dx \right)e^{i\zeta t} dt. 
\end{aligned}
\end{equation*}
The apparent singularity at \(\zeta=i\) is removable. Since
\(\mathcal U=\mathcal F U\), the left-hand side of the preceding
identity is $(U\mathcal C U^{-1}\mathcal F^{-1}f)(\zeta)$.
It follows that
\begin{align*}
(\mathcal U \mathcal C \mathcal U^{-1}f)(t)
=\int_t^\infty e^{-(x-t)}\frac{f(x)}{x}\,dx.
\end{align*}

It was shown in \cite{BrownHalmosShields1965} that the adjoint $\mathcal C_\infty^*$ is given by
\begin{align}\label{eq:cinftyadjoint}
 (\mathcal C_\infty^*f)(s)=\int_s^\infty\frac{f(x)}{x}\,dx. 
\end{align}
For \(v\in L^{2}(\mathbb R_+)\), the definition of $W_{k_{R}}$ gives
\begin{align*}
 (W_{k_{R}}v)(t)
 =\int_0^\infty k_{R}(t-x)v(x)\,dx
 =\int_t^\infty e^{-(x-t)}v(x)\,dx.
\end{align*}
Using Fubini's theorem again, we obtain
\begin{equation*} 
\begin{aligned}
(W_{k_R}\mathcal C_\infty^*f)(t)
&=\int_t^\infty e^{-(s-t)}(\mathcal C_\infty^*f)(s)\,ds\\ 
&=\int_t^\infty e^{-(s-t)} \left(\int_s^\infty\frac{f(x)}{x}\,dx\right)ds\\ 
&=\int_t^\infty \left(\int_t^x e^{-(s-t)}\,ds\right) \frac{f(x)}{x}\,dx\\
&=\int_t^\infty \bigl(1-e^{-(x-t)}\bigr)\frac{f(x)}{x}\,dx.
\end{aligned}
\end{equation*}
Consequently,
$(I-W_{k_R})\mathcal C_\infty^*f
=\mathcal U \mathcal C \mathcal U^{-1}f.$
The operator $W_{k_R}$ is bounded because $k_R\in L^1(\mathbb R)$,
and $\mathcal C_\infty^*$ is bounded by
\cite{BrownHalmosShields1965}.
Since \(C_c^\infty(\mathbb R_+)\) is dense in
\(L^2(\mathbb R_+)\), and both sides define bounded operators,
the identity extends to all of \(L^2(\mathbb R_+)\). Thus,
\begin{align*}
\mathcal U \mathcal C \mathcal U^{-1}=(I-W_{k_R})\mathcal C^*_\infty.
\end{align*}
\end{proof}
Recall the unitary Mellin transform $\mathbf M$ (see \cite[p.166]{Nikolski2002} or \cite[p.748]{Yafaev2017}) defined by
\begin{align}\label{mellin}
 \mathbf M:L^2(\mathbb R_+)\longrightarrow L^2(\R),
 \qquad
 (\mathbf M f)(s)=\frac{1}{\sqrt{2\pi}}
 \int_0^\infty x^{-1/2-is}f(x)\,dx.
\end{align}
Note that 
$\mathbf M=\mathcal F V$,
where
\begin{align*}
(Vf)(t)&=e^{t/2}f(e^t):L^2(\mathbb R_{+})\to L^2(\mathbb R)\ \text{ is unitary}.\\
(V^{-1}g)(x)&=x^{-1/2}g(\log x):L^2(\mathbb R)\to L^2(\mathbb R_{+}).
\end{align*}
For $f\in L^{\infty}(\mathbb R),$ define 
\begin{align*}
\Lambda_{f}=\mathbf M^{-1}M_f\mathbf M: L^2(\mathbb R_{+})\to L^2(\mathbb R_{+}).
\end{align*}
Since $M_f M_g=M_{fg},\Lambda_{f}\Lambda_{g}=\Lambda_{fg}.$

The truncated Fourier transform $\mathcal F_{+}$ on $L^2(\mathbb R_{+})$ is defined by
\begin{align*}
(\mathcal F_{+}f)(t)=\frac{1}{\sqrt{2\pi}}\int_{\mathbb R_{+}} f(x)e^{-ixt}dx, \quad t>0.
\end{align*}
Thus, $\mathcal F_{+}=P^{+}\mathcal F P^{+}.$ Its adjoint operator
is given by
\begin{align*}
(\mathcal F^{*}_{+}f)(t)=\frac{1}{\sqrt{2\pi}}\int_{\mathbb R_{+}} f(x)e^{ixt}dx, \quad t>0.
\end{align*}
Using Proposition \ref{WFtau}(2), we have the following Lemma.
\begin{lemma}\label{lemma:WFF}
If $\Phi\in L^{\infty}(\mathbb R),$ then $W_{\widehat\Phi} = \left.P^+\mathcal F M_\Phi\mathcal F^*P^+ \right|_{L^2(\mathbb R_+)}.$
\end{lemma}
\begin{lemma}\label{lemma:R+}
Define the operator 
$(Rf)(x)=x^{-1}f(x^{-1}),$
and  the function
\begin{align}
u (s)=\frac{\Gamma(\frac12-is)}{\sqrt{2\pi}}
       e^{i\pi/4+\pi s/2}. 
\end{align}
Then 
\begin{enumerate}
  \item $\mathcal{F}^{*}_{+}=\Lambda_{u}R;$ 
  \item $W_{\hat{\chi}_{\mathbb R{+}}}=\Lambda_{|u(-\cdot)|^2},$ where $|u(s)|^2=\frac{1}{1+e^{-2\pi s}}.$
  \item $\sigma(W_{\hat{\chi}_{\mathbb R{+}}})=[0,1].$
\end{enumerate}

\end{lemma}
\begin{proof}
Katsnelson \cite[(3.6)]{Katsnelson2019} proved that 
\begin{align*}
(\mathbf M\mathcal F^{*}_{+}f)(s) = \frac{\Gamma(\frac12-is)}{\sqrt{2\pi}}\, e^{i\pi/4+\pi s/2}\, (\mathbf Mf)(-s).
\end{align*}
A direct calculation gives 
\begin{align*}
(\mathbf MRf)(t)&=\frac{1}{\sqrt{2\pi}}\int_{0}^{\infty}x^{-(1/2+it)}x^{-1}f(x^{-1})dx\\
&\overset{x^{-1}=s}{=}\frac{1}{\sqrt{2\pi}}\int_{\infty}^{0}s^{1/2+it}sf(s)(-s^{-2})ds\\
&=\frac{1}{\sqrt{2\pi}}\int_{0}^{\infty}s^{-1/2+it}f(s)ds\\
&=(\mathbf Mf)(-t).
\end{align*}
Thus, $\mathcal{F}^{*}_{+}=\Lambda_{u}R.$ 
Since $R$ is self-adjoint unitary on $L^2(\mathbb R_{+})$,
Lemma \ref{lemma:WFF} implies that
\begin{align*}
W_{\hat{\chi}_{\mathbb R{+}}}=\mathcal{F}_{+}M_{\chi_{\mathbb R{+}}}\mathcal{F}^{*}_{+}=\mathcal{F}_{+}\mathcal{F}^{*}_{+}=R\Lambda_{|u|^2}R=\Lambda_{|u(-\cdot)|^2}.
\end{align*}
Using the formula $\left|\Gamma\!\left(\frac12-is\right)\right|^2=\frac{\pi}{\cosh(\pi s)},$ we have 
\begin{align}\label{gs}
0<|u(s)|^2=\frac{|\Gamma(\frac12-is)|^2}{2\pi}
       |e^{i\pi/4+\pi s/2}|^2=\frac{e^{\pi s}}{2\cosh \pi s}=\frac{1}{1+e^{-2\pi s}}<1.
\end{align}
Since the spectrum of a multiplication operator is  the  essential range of its symbol, $\sigma(W_{\hat{\chi}_{\mathbb R{+}}})=[0,1].$
\end{proof}
Gallardo-Guti{\'e}rrez, Partington, and Ross proved in
\cite[Proposition~3.14]{GallardoGutierrezPartingtonRoss2025Hardy} that
$I-\mathcal C^{*}_{\infty}$ is unitarily equivalent to
 $M_{\varphi}$ on $L^{2}(i\mathbb R)$ via unitary Mellin transform, where
$\varphi(s)=\frac{s-1/2}{s+1/2}$. Hence, we have the following Lemma.
\begin{lemma}\label{lemma:cinfty}
Let $\psi(s)=(\frac12-is)^{-1}.$ Then $\mathcal C^*_\infty=\Lambda_{\psi}.$ 
\end{lemma}

\subsection{Proof of Theorem \ref{thm:main1}}
\begin{proof}[Proof of Theorem~\ref{thm:main1}]
In view of \eqref{fenjie}, 
we need to compute separately 
the representations of $\mathcal C^*_\infty$ and $W_{k_R}$ 
on $H^2(\mathbb D)$ under the unitary operator $\mathcal U$.

Since 
$\lim_{s\to 0^{+}}\varphi(s)=0$ and $
\lim_{s\to 1^{-}}\varphi(s)=0,$ $\varphi$ is continuous on 
$[0,1]$. By Lemma \ref{lemma:R+}, we have
$\sigma(W_{\hat{\chi}_{\mathbb R{+}}})=[0,1]$. 
A computation gives $\varphi (\frac{1}{1+e^{2\pi s}})=\frac{1}{2}\psi(s).$
Using continuous functional calculus and
Lemma \ref{lemma:cinfty}, we obtain
\begin{align*}
2\mathbf M \varphi(W_{\hat{\chi}_{\mathbb R{+}}})\mathbf M^{-1}
=2\varphi(M_{\frac{1}{1+e^{2\pi s}}})=2M_{\varphi (\frac{1}{1+e^{2\pi s}})}=M_{\psi}
=\mathbf M\mathcal C_\infty^*\mathbf M^{-1}.
\end{align*}
Hence, $\mathcal C^*_\infty=2\varphi(W_{\hat{\chi}_{\mathbb R{+}}})$.

Since 
$\omega^{-1}(e^{i\theta})=-\cot(\theta/2)$, we have
$\chi_{\mathbb R_+}\circ\omega^{-1}=\chi_{\mathbb T_-}$. Proposition~\ref{pr:wiener} implies that
$
W_{\hat{\chi}_{\mathbb R{+}}}
=
\mathcal U
T_{\mathord{\scalebox{1}{$\chi$}}_{\mathbb T_-}}
\mathcal U^{-1}.$
By continuous functional calculus,
\begin{align}\label{cinfty}
  \mathcal C_\infty^*
=
2\mathcal U
\varphi\left(
T_{\mathord{\scalebox{1}{$\chi$}}_{\mathbb T_-}}
\right)
\mathcal U^{-1}.
\end{align}

Recall $k_{R}(y)=e^{y}\chi_{\mathbb R_{-}}(y)$ in Lemma \ref{2u}.
By Proposition \ref{pr:wiener}, we have
\begin{equation*}
(\sqrt{2\pi}\mathcal{F}^{-1}k_R)(x)=\int_{-\infty}^{0}e^{(1+ix)t}dt
 =\frac{1}{1+i x}, \quad x\in \mathbb{R}
\end{equation*}
and 
\begin{align*}
((\sqrt{2\pi}\mathcal{F}^{-1}k_R)\circ \omega^{-1})(\xi)=\frac{1}{2}(1-\frac{1}{\xi}),
\qquad \xi\in \mathbb T\setminus\{1\}.
\end{align*}
Since $1/\xi=\overline\xi$ on $\mathbb T$, Proposition~\ref{pr:wiener}
implies that
$W_{k_R}=\frac{1}{2}\mathcal U(I-S^*)\mathcal U^{-1}$, and hence
\begin{align}\label{wkr}
I-W_{k_R}=\frac{1}{2}\mathcal U(I+S^*)\mathcal U^{-1}.
\end{align}
Combining \eqref{cinfty}, \eqref{wkr} and \eqref{fenjie}, we obtain
$\mathcal U\mathcal C\mathcal U^{-1}
=\mathcal U(I+S^*)
\varphi(T_{\mathord{\scalebox{1}{$\chi$}}_{\mathbb T_-}})
\mathcal U^{-1}.$
Hence,
\begin{align}\label{cg}
\mathcal C=(I+S^*)
\varphi(T_{\mathord{\scalebox{1}{$\chi$}}_{\mathbb T_-}}).
\end{align}
Since $T_{\mathord{\scalebox{1}{$\chi$}}_{\mathbb T_-}}\in \mathbf T$, 
$\varphi(T_{\mathord{\scalebox{1}{$\chi$}}_{\mathbb T_-}})\in\mathbf T.$
\eqref{cg} implies that $\mathcal C\in\mathbf T$.

Let $\rho:\mathbf T^\infty\longrightarrow L^\infty(\mathbb T)
$ be the symbol map. Barr\'{\i}a and
Halmos \cite{BarriaHalmos1982}
showed that $\rho(\mathcal C)=0$ 
 and 
$\ker\!\left(\rho|_{\mathbf T}\right)=\mathbf Q.$
Since $\mathcal C\in\mathbf T$, it follows that $\mathcal C\in\mathbf Q$.
\end{proof}
\begin{remark}
The identity $\rho(\mathcal C)=0$ can also be derived from the definition of  $\varphi$.
\begin{align}\label{symbolzero}
  \rho\!\left(\varphi(T_{\chi_{\mathbb T_-}})\right)
  =\varphi(\chi_{\mathbb T_-})=0,
\end{align}
because $\varphi(0)=\varphi(1)=0$.  Thus,
$\rho(\mathcal C)=\rho((1+S^*))\rho\!\left(\varphi(T_{\chi_{\mathbb T_-}})\right)=0.$
\end{remark}

\section{Weighted Hardy operator}

Let $b$ be a bounded Lipschitz function on $\mathbb R.$ 
Define the operator
\begin{align}\label{cb-infty}
(\mathcal C_bf)(x)
=
\frac{b(\log x)}{x}\int_0^x f(y)\,dy,
\qquad f\in L^2(\mathbb R_{+}).
\end{align}
Since $\mathcal{C}_b=M_{b\circ log}\mathcal{C}_{\infty},$
$\mathcal{C}_b$ is a bounded operator on $L^2(\mathbb R_{+}).$
Define the operator $E:H^2(\mathbb T)\to L^2(\mathbb R_{+})$ as
\begin{align*}
Ez^{n-1}=\chi_{(n-1,n]}\quad n\geq 1.
\end{align*}
Thus, $E$ is an isometry and $EE^{*}$ is an orthogonal projection. Moreover,
\begin{align*}
(EE^{*}f)(t)=\int_{n-1}^{n}f(x)dx,\qquad t\in (n-1,n].
\end{align*}
The operator $E^{*}\mathcal C_b E$ has the following matrix representation with respect to the basis $\{z^{n-1}\}_{n\geq 1}:$
\begingroup
\renewcommand{\arraystretch}{1}
\setlength{\arraycolsep}{6pt}
\begin{equation}\label{eq:cinftymatrix}
\begin{pmatrix}
d_1 & \lightzero & \lightzero & \lightzero & \cdots \\
c_2 & d_2 & \lightzero & \lightzero & \cdots \\
c_3 & c_3 & d_3 & \lightzero & \cdots \\
c_4 & c_4 & c_4 & d_4 & \cdots \\
\vdots & \vdots & \vdots & \vdots & \ddots
\end{pmatrix},
\end{equation}
\endgroup
where
$d_n=\int_{n-1}^{n}b(\log x)\frac{x-n+1}{x}dx(n\geq 1)$ and $ c_m=\int_{m-1}^{m}\frac{b(\log x)}{x}dx(m\geq 2).$
If we set all the diagonal entries of the matrix \eqref{eq:cinftymatrix} to zero and then delete  the first row, we obtain a Rhaly matrix.

If $b\equiv1,$ then $\mathcal C_b=\mathcal C_{\infty}$. Moreover, $c_n=\log\frac{n}{n-1}$ and $d_n-\frac{1}{n}=(n-1)(\frac{1}{n}-c_n).$
It follows that 
\begin{align*}
\|E^*\mathcal C_{\infty}E-\mathcal C\|^2_{HS}
=\sum_{n=2}^{\infty}((n-1)+(n-1)^2)(c_n-\frac{1}{n})^2
\leq\sum_{n=2}^{\infty}\frac{1}{n(n-1)}<\infty,
\end{align*}
where $\|\cdot\|_{HS}$ norm denotes the Hilbert-Schmidt norm.
Hence, we have the following lemma.
\begin{lemma}\label{JCJC}
The operator $E^*\mathcal C_{\infty}E-\mathcal C$ is Hilbert-Schmidt.
\end{lemma}
For two bounded operators $A$ and $B,$ their commutator is defined by $[A,B]=AB-BA.$ 
\begin{lemma}\label{IJJ}
The operators 
\begin{align*}
(I-EE^{*})\mathcal C_{\infty}E,\quad (I-EE^{*})\mathcal C^{*}_{\infty}E,\quad [EE^{*},\mathcal C_{\infty}]\quad \text{and} \quad [EE^{*},\mathcal C^{*}_{\infty}]
\end{align*}
are Hilbert-Schmidt.
\end{lemma}
\begin{proof}
Set $d_1=1$ and
$d_n=1-(n-1)\log\frac{n}{n-1}$ for $n\ge2$.
For $n\ge1$, we have 
\begin{align*}
(\mathcal C_\infty\chi_{(n-1,n]})(t)
&=
\begin{cases}
0, & 0<t\le n-1,\\[0.5ex]
1-\dfrac{n-1}{t}, & n-1<t\le n,\\[0.5ex]
\dfrac1t, & t>n,
\end{cases}
\end{align*}
and
\begin{align*}
(EE^*\mathcal C_\infty\chi_{(n-1,n]})(t)
&=
\begin{cases}
0, & 0<t\le n-1,\\[0.5ex]
d_n, & n-1<t\le n,\\[0.5ex]
\displaystyle\log\frac{m}{m-1},
& t\in(m-1,m],\quad m>n.
\end{cases}
\end{align*}
Let $h_n(t)=\frac{1}{t}-\log\frac{n}{n-1},\ t\in (n-1,n],\ n\geq2,$ we have 
\begin{align*}
\|(I-EE^{*})\mathcal C_{\infty}E\|^2_{HS}
&=\sum_{n=1}^{\infty}\|(I-EE^{*})\mathcal C_{\infty}Ez^{n-1}\|^2\\
&=\sum_{n=1}^{\infty}\|(I-EE^{*})\mathcal C_{\infty}\chi_{(n-1,n]}\|^2\\
&=\sum_{n=1}^{\infty}\sum_{m=2}^{\infty}\|(I-EE^{*})\mathcal C_{\infty}\chi_{(n-1,n]}\|^2_{L^2(m-1,m]}\\
&=\sum_{m=2}^{\infty}\sum_{n=1}^{\infty}\|(I-EE^{*})\mathcal C_{\infty}\chi_{(n-1,n]}\|^2_{L^2(m-1,m]}\\
&=\sum_{m=2}^{\infty}\left(\sum_{n=1}^{m-1}\|h_m\|^2_{L^2(m-1,m]}+(m-1)^2\|h_m\|_{L^2(m-1,m]}^2\right)\\
&=\sum_{m=2}^{\infty}m(m-1)\|h_m\|^2_{L^2(m-1,m]}.
\end{align*}
Since $|h_m(t)|\leq \frac{1}{m-1}-\frac{1}{m}=\frac{1}{m(m-1)},$ we have 
$\|(I-EE^{*})\mathcal C_{\infty}E\|^2_{HS}\leq \sum_{m=2}^{\infty}\frac{1}{m(m-1)}<\infty.$ It follows that $(I-EE^{*})\mathcal C_{\infty}E$ is Hilbert-Schmidt.
For $n\geq 2,$ by the formula \eqref{eq:cinftyadjoint}, we have
\begin{align*}
(\mathcal C^{*}_{\infty}\chi_{(n-1,n]})(t)=
\begin{cases} 
\log\frac{n}{n-1}, & t\in (0,n-1];\\ 
\log\frac{n}{t}, & t\in (n-1,n];\\
0, & t>n,
\end{cases}
\end{align*}
and 
\begin{align*}
\|(I-EE^{*})\mathcal C^{*}_{\infty}\chi_{(n-1,n]}\|^2_2&=\int_{n-1}^{n}
\left(\log\frac{n}{t}-\int_{n-1}^{n}\log\frac{n}{s}ds\right)^2dt\\
&\leq 
\int_{n-1}^{n}\log^2\frac{n}{t}dt\leq \log^{2}\frac{n}{n-1}
\leq \frac{1}{(n-1)^2}.
\end{align*}
Note that  $\mathcal C^{*}_{\infty}\chi_{(0,1]}=\chi_{(0,1]}\log\frac{1}{t}$
and  \[\|(I-EE^{*})\mathcal C^{*}_{\infty}\chi_{(0,1]}\|^2_2=\int_{0}^{1}\left(\log x+1\right)^2 dx=1.\] So,
$\|(I-EE^{*})\mathcal C^{*}_{\infty}E\|^2_{HS}\leq 1+\sum_{m=2}^{\infty}\frac{1}{(m-1)^2}<\infty.$ It follows that $(I-EE^{*})\mathcal C^{*}_{\infty}E$ is Hilbert-Schmidt.

Since the family of all Hilbert-Schmidt operators is a $*$-ideal, $(I-EE^{*})\mathcal C_{\infty}EE^{*}$ and $(I-EE^{*})\mathcal C^{*}_{\infty}EE^{*}$ are Hilbert-Schmidt. Moreover,
\begin{align}
[EE^{*},\mathcal C_{\infty}]=EE^{*}\mathcal C_{\infty}(I-EE^{*})-(I-EE^{*})\mathcal C_{\infty}EE^{*}
\end{align}
and $[EE^{*},\mathcal C^{*}_{\infty}]$
are Hilbert-Schmidt. 
\end{proof}

\begin{lemma}\label{lemma:c0}
If $f\in C_0(\mathbb R),$ then \ $E^{*} \Lambda_f E \in \mathbf Q$ and $[EE^*,\Lambda_f]$ is compact.
\end{lemma}
\begin{proof}
Recall the function $\bar{\psi}(s)=\frac{1}{1/2+i s}$ in Lemma \ref{lemma:cinfty}. 
Note that $\bar{\psi}:\mathbb R\to \{z:|z-1|=1,z\neq 0\}$
 is a homeomorphism.
Since $f\in C_0(\mathbb R),$
define $F(z)=f(\bar{\psi}^{-1}(z))$ and $F(0)=0.$ It follows that $F$ is continuous on $\{z:|z-1|=1\}.$
By the Stone-Weierstrass Theorem, there exist polynomials $p_n(z,\bar{z})$ with $p_n(0,0)=0$ such that $p_n$ converges uniformly to $F$ on 
$\{z:|z-1|=1\}.$ 
Write $p_n(\mathcal C_{\infty},\mathcal C^{*}_{\infty})=\Lambda_{p_n(\bar{\psi},\psi)}.$
It follows that 
\begin{align}\label{lamf}
\|\Lambda_f - p_n(\mathcal C_{\infty},\mathcal C^{*}_{\infty})\|\to 0.
\end{align}
Since $[EE^{*},\mathcal C_{\infty}]$ and $[EE^{*},\mathcal C^{*}_{\infty}]$ are compact, $[EE^{*},p_n(\mathcal C_{\infty},\mathcal C^{*}_{\infty})]$ is compact. Moreover,
$[EE^{*},\Lambda_f]$ is compact.

By Theorem \ref{thm:main1} and Lemma \ref{JCJC}, we have $E^*\mathcal C_{\infty}E$
and $E^*\mathcal C^{*}_{\infty}E$ belong to $\mathbf Q.$ Moreover, $p_n(E^*\mathcal C_{\infty}E,E^*\mathcal C^*_{\infty}E)\in \mathbf Q.$ 
Repeated application of  Lemma \ref{IJJ} shows
that
\[E^*p_n(\mathcal C_{\infty},\mathcal C^{*}_{\infty})E-p_n(E^*\mathcal C_{\infty}E,E^*\mathcal C^{*}_{\infty}E)\] is compact.
To clarify this point, we give the following calculation.
If $p_n(z,\bar{z})=z\bar{z}^2,$ then
\begin{align*}
&E^*\mathcal C_{\infty}\mathcal C^{*2}_{\infty}E-(E^*\mathcal C_{\infty}E)(E^*\mathcal C^{*}_{\infty}E)^2\\
&\quad=E^*\mathcal C_{\infty}\mathcal C^{*2}_{\infty}E-E^*\mathcal C_{\infty}\mathcal C^{*}_{\infty}E(E^*\mathcal C^*_{\infty}E)\\
&\quad+E^*\mathcal C_{\infty}\mathcal C^{*}_{\infty}E(E^*\mathcal C^*_{\infty}E)-
(E^*\mathcal C_{\infty}E)(E^*\mathcal C^{*}_{\infty}E)^2\\
&\quad=E^*\mathcal C_{\infty}\mathcal C^{*}_{\infty}(I-EE^{*})\mathcal C^{*}_{\infty}E
+E^*\mathcal C_{\infty}(I-EE^*)\mathcal C^{*}_{\infty}EE^*\mathcal C^*_{\infty}E.
\end{align*}
By Lemma \ref{IJJ}, the above operator is compact. Since $E$ is an isometry, \eqref{lamf} implies that $E^{*} \Lambda_f E \in \mathbf Q.$
\end{proof}
\begin{lemma}\label{lemma:toeplitz}
If $f\in L^{\infty}(\mathbb R),$ then $E^* W_{\hat{f}}E$ is a Toeplitz operator.
\end{lemma}
\begin{proof}
The definitions of $E$ and the translation $\tau_1$ imply that
$ES=\tau_1 E.$
By Proposition \ref{WFtau}, we have 
$\tau^{*}_1W_{\hat{f}}\tau_1=W_{\hat{f}}.$ It follows that 
\begin{align*}
S^*(E^* W_{\hat{f}}E)S=E^* \tau^{*}_1W_{\hat{f}}\tau_1 E
=E^* W_{\hat{f}}E.
\end{align*}
By \cite[Theorem 6]{BrownHalmos1964}, $E^* W_{\hat{f}}E$ is a Toeplitz operator.
\end{proof}
\begin{lemma}\label{jwj}
If $h\in L^{\infty}(\mathbb R)$ and $f,g\in C_0(\mathbb R),$ then $E^{*} \Lambda_{f}W_{\hat{h}}\Lambda_{g}E \in \mathbf Q.$
\end{lemma}
\begin{proof}
A direct calculation gives
\begin{equation}
  \begin{aligned}
    &E^*\Lambda_fW_{\hat{h}}\Lambda_gE-(E^*\Lambda_fE) (E^*W_{\hat{h}}E)(E^*\Lambda_gE)\\
&\quad=E^*\Lambda_f(I-EE^*)W_{\hat{h}}\Lambda_gE
       +E^*\Lambda_fEE^*W_{\hat{h}}(I-EE^*)\Lambda_gE.
  \end{aligned}
\end{equation}
Lemma \ref{lemma:c0} implies that  $E^*\Lambda_f(I-EE^*)=E^{*}[EE^*,\Lambda_f]$ is compact. Similarly, $(I-EE^*)\Lambda_gE$ is compact. 

Using Lemma \ref{lemma:c0} and Lemma \ref{lemma:toeplitz}, we have $(E^*\Lambda_fE) (E^*W_{\hat{h}}E)(E^*\Lambda_gE)$ belong to $\mathbf Q.$ Therefore, $E^{*} \Lambda_{f}W_{\hat{h}}\Lambda_{g}E \in \mathbf Q.$
\end{proof}

\begin{lemma}\label{jmj}
If $h\in L^{\infty}(\mathbb R)$ and $f,g\in C_0(\mathbb R),$ then
$E^{*} \Lambda_{f}M_{h}\Lambda_{g}E \in \mathbf Q.$
\end{lemma}
\begin{proof}
Suppose that $f,g\in C_{c}(\mathbb R).$ 
Recall the function $u (s)=\frac{\Gamma(\frac12-is)}{\sqrt{2\pi}}e^{i\pi/4+\pi s/2}$ in Lemma \ref{lemma:R+}. Since $u$ is continuous and has no zeros,
$f/u$ and $g/\bar{u}$ are in $C_{c}(\mathbb R)\subset C_{0}(\mathbb R).$
Using Lemma \ref{lemma:R+} and Lemma \ref{lemma:WFF}, we have
\begin{align}
\Lambda_f M_h\Lambda_g
&=\Lambda_{f/u}\Lambda_u M_h\Lambda_{\bar u}\Lambda_{g/\bar u}
 =\Lambda_{f/u}\mathcal{F}^{*}_{+}R M_h R\mathcal{F}_{+}
   \Lambda_{g/\bar u} \notag\\
&=\Lambda_{f/u}\mathcal{F}^{*}_{+} M_{h(1/\cdot)}\mathcal{F}_{+}
   \Lambda_{g/\bar u}
 =\Lambda_{f/u}W_{\hat{\tilde{h}}}\Lambda_{g/\bar u},
\end{align}
where
\[\tilde h(x)=
 \begin{cases}
 h(-1/x),&x<0,\\
 0,&x\ge0.
 \end{cases}\]
 By Lemma \ref{jwj}, we obtain 
$E^*\Lambda_{f}M_{h}\Lambda_{g}E \in \mathbf Q.$

For general $f,g\in C_{0}(\mathbb R).$ Let $k_n\in C_c(\mathbb R)$ such that $0\leq k_n\leq1$, $k_n=1$ on $[-n,n]$, and $k_n=0$ outside $[-n-1,n+1]$. On the remaining intervals $[-n-1,-n]$ and $[n,n+1]$, define $k_n$ by linear functions. Note that
\begin{align*}
\|fk_n-f\|_\infty \le\sup_{|x|>n}|f(x)| \longrightarrow0 \quad \text{and}\quad
\|gk_m-g\|_\infty \le\sup_{|x|>m}|g(x)| \longrightarrow0.
\end{align*}
Therefore, $fk_n,gk_m\in C_{c}(\mathbb R).$
Moreover,
\begin{align*}
\|\Lambda_{f}M_{h}\Lambda_{g}-\Lambda_{fk_n}M_{h}\Lambda_{gk_m}\|
=&\|\Lambda_{f}M_{h}\Lambda_{g}-\Lambda_{fk_n}M_{h}\Lambda_{g}+\Lambda_{fk_n}M_{h}\Lambda_{g}-\Lambda_{fk_n}M_{h}\Lambda_{gk_m}\|\\
\leq &\|f-fk_n\|_{\infty}\|h\|_{\infty}\|g\|_{\infty}+
\|fk_n\|_{\infty}\|h\|_{\infty}\|g-gk_{m}\|_{\infty}.
\end{align*}
Since $\mathbf Q$ is norm closed, $E^{*} \Lambda_{f}M_{h}\Lambda_{g}E \in \mathbf Q.$
\end{proof}
\begin{lemma}\label{lemma:Ln}
Let $L$ be the Lipschitz constant of $b.$ For every $n,$ define
$\phi_n(t)=\left(1+t^2/n^2\right)^{-1}.$
Then 
\begin{align*}
\|[M_{b\mathbin{\scriptscriptstyle\circ}\log},\Lambda_{\phi_n}]\|\leq L/n.
\end{align*}
\end{lemma} 
\begin{proof}
A direct calculation yields 
$VM_{b\mathbin{\scriptscriptstyle\circ}\log}V^{-1}=M_{b}.$
Since $\mathbf M=\mathcal F V$, we have \[V[M_{b\mathbin{\scriptscriptstyle\circ}\log},\Lambda_{\phi_n}]V^{-1}=[M_b,\mathcal F^{-1} M_{\phi_n} \mathcal F].\]
For $f\in L^2(\mathbb R),$
using the convolution theorem, we have
\begin{align*}
\mathcal F^{-1}M_{\phi_n}\mathcal F f
=\kappa_n\ast f,
\end{align*}
where
$\kappa_n(w)=\frac1{\sqrt{2\pi}}(\mathcal F^{-1}\phi_n)(w)
=\frac n2\,e^{-n|w|}.$
Since $b$ is Lipschitz, $|b(t-x)-b(t)|\leq L|x|.$ Thus,
\begin{align*}
|([M_b,\mathcal F^{-1} M_{\phi_n} \mathcal F]f)(t)|& \leq \int_{\mathbb R} 
\kappa_n(x)|b(t-x)-b(t)||f(t-x)|dx \\
& \leq  L\int_{\mathbb R} 
\kappa_n(x)|x||f(t-x)|dx.
\end{align*}
Using Young's convolution inequality, we have 
\begin{align*}
\|[M_b,\mathcal F^{-1} M_{\phi_n} \mathcal F]f\|_2\leq
 L\left(\int_{\mathbb R}|x|\kappa_n(x)dx\right)\|f\|_2.
\end{align*}
A direct calculation yields 
\[\int_{\mathbb R}|x|\kappa_n(x)dx=\frac{n}{2}\int_{\mathbb R}|x|e^{-n|x|}dx=\int_{\mathbb R_{+}}nxe^{-nx}dx=\frac{1}{n}.\]
\end{proof}
\begin{theorem}\label{the:cb}
For every bounded Lipschitz function $b$ on $\mathbb{R},$ we have   
$E^{*} \mathcal C_b E \in \mathbf Q.$ 
\end{theorem}
\begin{proof}
By Lemma \ref{lemma:cinfty}, we have $\mathcal C_{\infty}=\Lambda_{\bar{\psi}},$ and 
$\psi(s)=(\frac12-is)^{-1}\in C_{0}(\mathbb{R}).$ Thus, 
$\mathcal C_b=M_{b\mathbin{\scriptscriptstyle\circ}\log}\Lambda_{\bar \psi}.$
For every $n,$
\begin{equation}\label{eq:split}
\mathcal C_b=\Lambda_{\phi_n}M_{b\mathbin{\scriptscriptstyle\circ}\log}\Lambda_{\bar \psi}
+[M_{b\mathbin{\scriptscriptstyle\circ}\log},\Lambda_{\phi_n}]\,\Lambda_{\bar \psi}\;+\;M_{b\mathbin{\scriptscriptstyle\circ}\log}\,\Lambda_{(1-\phi_n)\bar \psi},
\end{equation}
where $\phi_n(t)=\left(1+t^2/n^2\right)^{-1}\in C_{0}(\mathbb R).$
Note that 
\begin{align*}
|1-\phi_{n}(t)||\bar{\psi}(t)|=\frac{|t|^2}{(n^2+t^2)\sqrt{1/4+t^2}}\leq \frac{|t|}{n^2+t^2}\leq \frac{1}{2n}.
\end{align*}
The first inequality follows from $\sqrt{1/4+t^2}\geq |t|,$ and the second inequality follows from $n^2+t^2\geq 2n|t|.$
Note $\|\psi\|_{\infty}=2.$ By Lemma \ref{lemma:Ln}, we have
\begin{align*}
\|\mathcal C_b-\Lambda_{\phi_n}M_{b\mathbin{\scriptscriptstyle\circ}\log}\Lambda_{\bar \psi}\|
\leq \frac{2L+\|b\|_{\infty}/2}{n}\to 0, \qquad n\to \infty.
\end{align*}
Lemma \ref{jmj} implies that $E^* \Lambda_{\phi_n}M_{b\mathbin{\scriptscriptstyle\circ}\log}\Lambda_{\bar \psi} E\in \mathbf Q.$ Since $\mathbf Q$ is norm closed, $E^{*} \mathcal C_b E\in \mathbf Q.$
\end{proof}
\section{A necessary condition for an operator to belong to the Hankel algebra}
Define 
\begin{align*}
\eta_{N}(e^{it})=\frac{1}{\sqrt{N}}\sum_{k=0}^{N-1}e^{ikt}.
\end{align*}
Note that $|\eta_{N}|^2$ is the Fejér kernel,
$\|\eta_N\|=1$ and 
\begin{align}
|\eta_{N}(e^{it})|^2=\frac{1}{N}\frac{|1-e^{iNt}|^2}{|1-e^{it}|^2}
=\frac{1}{N}\frac{\sin^2(Nt/2)}{\sin^2(t/2)}\leq \frac{1}{N}\frac{1}{t^2/\pi^2}.
\end{align}
Let
$E_{L,N}=\{e^{it}:|t|\leq\min\{L/N,\pi\}\}.$ If $L/N\geq \pi,E_{L,N}=\mathbb T.$ In general, we only consider the case where $L/N < \pi$.
Let $M_{L,N}$ denote
the multiplication operator $M_{\chi_{E_{L,N}}}$ on $L^{2}(\mathbb T).$ 
\begin{align}\label{ln}
\|(I-M_{L,N})\eta_{N}\|^2\leq \frac{1}{2\pi}\int_{L/N<|t|\leq\pi}|\eta_{N}(e^{it})|^2dt\leq \frac{1}{2\pi}\frac{\pi^2}{N}\frac{2N}{L}=\frac{\pi}{L}\to 0,\quad \text{as}\ L\to\infty.
\end{align}
\begin{definition}
We call a sequence $\{f_N\}_{N\geq1}\subset L^2(\mathbb T)$ localized if $\sup_{N\geq 1}\|f_N\|<\infty$ and 
\begin{align*}
\lim_{L\to\infty}\sup_{N\geq 1}\|(I-M_{L,N})f_N\|=0.
\end{align*}
\end{definition}
\begin{lemma}\label{fejer}
The sequence $\{\eta_N\}_{N\geq 1}$ is localized.
\end{lemma}
\begin{lemma}\label{truncatedp}
Let $f\in L^2(\mathbb T)$ and its support is contained in $\{e^{is}:|s|\leq \delta\}, 0<\delta<\pi.$ Then
\begin{align*}
(Pf)(e^{it}) =\frac1{2\pi}\int_{-\delta}^{\delta} \frac{f(e^{is})}{1-e^{i(t-s)}}\,ds  \qquad\text{a.e. }\delta<|t|\le\pi.
\end{align*}
\end{lemma}
\begin{proof}
Recall the formula $Pf=\frac12(\int_{\mathbb T}f\frac{d\theta}{2\pi}+f+i\tilde f)$ (see \cite[p.104]{Garnett2007}), where
$\tilde f(e^{it}) =\frac1{2\pi}\operatorname{p.v.} \int_{-\pi}^{\pi} f(e^{is})\cot\frac{t-s}{2}ds.$ Then
$$(Pf)(e^{it}) =\frac1{4\pi}\int_{-\delta}^{\delta} f(e^{is})\left(1+i\cot\frac{t-s}{2}\right)\,ds\qquad\text{a.e. }\delta<|t|\le\pi.$$
Note that $\frac1{1-e^{i(t-s)}} =\frac12(1+i\cot\frac{t-s}{2}),\ t-s\notin2\pi\mathbb Z.$ Therefore,
$$(Pf)(e^{it}) =\frac1{2\pi}\int_{-\delta}^{\delta} \frac{f(e^{is})}{1-e^{i(t-s)}}\,ds  \qquad\text{a.e. }\delta<|t|\leq\pi.$$
\end{proof}
By the concavity of the sine function, we obtain the following lemma.
\begin{lemma}\label{sinconvex}
If $0< t_0<\pi,$ then $t\frac{\sin t_0}{t_0}\leq \sin t,\ t\in [0,t_0].$
\end{lemma}
\begin{lemma}\label{localized}
Let $\{f_N\}_{N\geq 1}\subset L^2(\mathbb T)$ be localized.
\begin{enumerate}
  \item If $\phi\in L^{\infty},$ then $\{\phi f_N\}_{N\geq 1}$ is localized.
  \item $\{J f_N\}_{N\geq 1}$ is localized.
  \item $\{P f_N\}_{N\geq 1}$ is localized.
\end{enumerate}
\end{lemma}
\begin{proof}
(1) Since $\sup_{N\geq 1}\|\phi f_N\|_2\leq \|\phi\|_{\infty}\sup_{N\geq 1}\|f_N\|<\infty$ and 
$$\|(I-M_{L,N})\phi f_N\|=\|\phi(I-M_{L,N})f_N\|\leq \|\phi\|_{\infty}\|(I-M_{L,N})f_N\|,$$
$\{\phi f_N\}_{N\geq 1}$ is localized.

(2) Since $J$ is unitary on $L^2(\mathbb T),$ 
$\sup_{N\geq1}\|Jf_N\|=\sup_{N\geq1}\|f_N\|<\infty.$
If $e^{it}\in E_{L,N},$ then $e^{-it}\in E_{L,N}.$ So $J(I-M_{L,N})=(I-M_{L,N})J$ and 
\begin{align*}
\|(I-M_{L,N})Jf_N\|=\|J(I-M_{L,N})f_N\|=\|(I-M_{L,N})f_N\|.
\end{align*}
Therefore, $\{J f_N\}_{N\geq 1}$ is localized.

(3) Since $\|P\|=1,$ $\sup_{N\geq1}\|Pf_N\|\leq\sup_{N\geq1}\|f_N\|<\infty.$
For $L_1\geq 4,$ take $L_2=\sqrt{L_1},$ so $L_1\geq 2L_2$ and $L_1/N>L_2/N,$
\begin{align*}
\|(I-M_{L_1,N})Pf_N\|&=\|(I-M_{L_1,N})Pf_N\|\\
&\leq \|(I-M_{L_1,N})PM_{L_2,N}f_N\|\\
&+\|(I-M_{L_1,N})P(I-M_{L_2,N})f_N\|.
\end{align*}
Since $\|I-M_{L_1,N}\|\leq 1, \|P\|=1$, we have
\begin{equation}\label{partA}
  \begin{aligned}
\|(I-M_{L_1,N})P(I-M_{L_2,N})f_N\|&\leq \|(I-M_{L_2,N})f_N\|.
  \end{aligned}
\end{equation}
If $L_1/N<|t|\leq \pi$ and $|v|\leq L_2/N,$ then $|v|\leq L_1/(2N)<|t|/2,$ it follows that $|t|/2\leq |t-v|\leq 3\pi/2.$ 
By Lemma \ref{sinconvex} on $[0, 3\pi/4]$, we have
$|1-e^{i(t-v)}|=2\sin(\frac{|t-v|}{2})\geq\frac{2\sqrt{2}}{3\pi}|t-v|\geq |t-v|/6 \geq |t|/12$ and 
\begin{align*}
|P(M_{L_2,N}f_N)(e^{it})|&\leq\frac{1}{2\pi}\int_{|v|\leq L_2/N}\left| \frac{f_N(e^{iv})}{1-e^{i(t-v)}}\right| dv\\
&\leq \frac{1}{2\pi}\frac{12}{|t|}\int_{|v|\leq L_2/N}\left|f_N(e^{iv})\right| dv\\
&\leq \frac{6}{\pi|t|}\left(\frac{2L_2}{N}\right)^{1/2}(2\pi)^{1/2}\|f_N\|.
\end{align*}
The first inequality follows from Lemma \ref{truncatedp}.
\begin{equation}\label{partB}
  \begin{aligned}
 \|(I-M_{L_1,N})PM_{L_2,N}f_N\|^2 &\leq\frac{1}{2\pi}\frac{36}{\pi^2}\frac{2L_2}{N}2\pi\|f_N\|^2\int_{L_1/N<|t|\leq \pi}\frac{1}{t^2}dt\\
&= \frac{72L_2}{\pi^2 N}\|f_N\|^2 2(N/L_1-\pi^{-1})\\
&\leq \frac{144}{\pi^2}\frac{L_2}{L_1}\|f_N\|^2\\
&= \frac{144}{\pi^2}\frac{1}{\sqrt{L_1}}\|f_N\|^2.
  \end{aligned}
\end{equation}
Combining \eqref{partA} with \eqref{partB}, we obtain $\{P f_N\}_{N\geq 1}$ is localized.
\end{proof}
\begin{theorem}\label{key2}
For every operator $T\in \mathbf T^{+},$
\begin{align*}
\lim_{L\to\infty}\sup_{N\geq1}\|(I-M_{L,N})T\eta_N\|_2=0.
\end{align*}
\end{theorem}
\begin{proof}
By Lemma \ref{fejer} and Lemma \ref{localized},
if $A$ is a
finite sum of finite products of Toeplitz operators and Hankel operators,
then
\begin{align*}
\lim_{L\to\infty}\sup_{N\geq1}\|(I-M_{L,N})A\eta_N\|_2=0.
\end{align*}
Let $T\in \mathbf T^{+}.$
For every $\varepsilon>0,$ there exists a finite sum of finite products of Toeplitz operators and Hankel operators $A_0$ such that $\|T-A_0\|<\frac{\varepsilon}{2}.$
\begin{align*}
\sup_{N\geq1} \|(I-M_{L,N})T\eta_N\|&\leq \sup_{N\geq1} \|(I-M_{L,N})(T-A_0)\eta_N\|+\sup_{N\geq1} \|(I-M_{L,N})A_0\eta_N\|\\
&\leq \|T-A_0\|+\sup_{N\geq1} \|(I-M_{L,N})A_0\eta_N\|.
\end{align*}
There exists $L_0>0$ such that \[\sup_{N\geq1} \|(I-M_{L,N})A_0\eta_N\|<\frac{\varepsilon}{2},\quad L\geq L_0.\]
Since $\varepsilon$ is arbitrary, $\lim_{L\to\infty}\sup_{N\geq1}\|(I-M_{L,N})T\eta_N\|_2=0.$
\end{proof}
\begin{corollary}\label{T+}
If a bounded Rhaly operator $\mathcal R_{a}\in \mathbf T^{+},$ then 
\[\lim_{\delta\to 0}
\sup_{\substack{N\geq1\\0\leq k\leq \delta N}}
\left(N\sum_{n=N}^\infty|a_{n+k}-a_n|^2\right)=0.\]
\end{corollary}
\begin{proof}
Let $\mathcal R_{a}\eta_{N}(z)=\sum_{n=1}^{\infty}y_n z^{n-1},$ we have
$$y_n=
 \begin{cases} \frac{n}{\sqrt{N}}a_n, & 1\leq n< N,\\ 
\sqrt{N}a_n, & n\geq N. \end{cases}$$
Thus,
\begin{align}\label{RaT}
\left(N\sum_{n\geq N}|a_{n+k}-a_n|^2\right)^{1/2}\leq \|(S^{*k}-I)\mathcal R_{a}\eta_{N}\|\leq \|(\bar{z}^k-1)\mathcal R_{a}\eta_{N}\|
\end{align}
On $E_{L,N},$ since $k\leq \delta N,|e^{-ikt}-1|\leq 2|\sin(\frac{kt}{2})|\leq k|t|\leq kL/N\leq \delta L.$
On $\mathbb T-E_{L,N},|e^{-ikt}-1|\leq 2.$
\begin{align*}
\|(\bar{z}^k-1)\mathcal R_{a}\eta_{N}\|
&\leq \|(\bar{z}^k-1)M_{L,N}\mathcal R_{a}\eta_{N}\|+
\|(\bar{z}^k-1)(I-M_{L,N})\mathcal R_{a}\eta_{N}\|\\
& \leq L\delta \|\mathcal R_{a}\|+2\|(I-M_{L,N})\mathcal R_{a}\eta_{N}\|\\
& \leq L\delta \|\mathcal R_{a}\|+2\sup_{N}\|(I-M_{L,N})\mathcal R_{a}\eta_{N}\|.
\end{align*}
By Theorem \ref{key2} and \eqref{RaT}, the conclusion holds.
\end{proof}
\section{Two Sequence spaces}
Define the following sequence space $\mathbf H.$
\begin{align*}
\mathbf H=\left\{a=(a_n)_{n\geq 1}:\|a\|_{\mathbf H}=\sup_{N\geq 1}\left(N\sum_{n=N}^{\infty}|a_n|^2\right)^{1/2}<\infty\right\}.
\end{align*}
Let $\mathbf{B}$ denote the set of all sequences $a = (a_n)_{n \geq 1}$ for which there exists a constant $B_{a}$ such that 
\[|a_n| \leq \frac{B_{a}}{n},\quad |a_{n+1} - a_n| \leq \frac{B_{a}}{n^2},\quad n \geq 1.\] 
For every $\delta>0,$ define
\begin{align*}
\omega_a(\delta)=\sup_{\substack{N\geq1\\0\leq k\leq \delta N}}
\left(N\sum_{n=N}^\infty|a_{n+k}-a_n|^2\right)^{1/2}.
\end{align*}
Note that $\omega_{a}(\delta)$ is decreasing as $\delta \to 0.$
\begin{proposition}\cite[Theorems 2.3 and 2.6]{BDS2025}\label{bddcompact}
Let  $a=(a_n)_{n \ge 1}$ be a complex sequence.
\begin{enumerate}
  \item The Rhaly operator $\mathcal R_{a}$ is bounded on $H^2(\mathbb T)$ if and only if 
$\|a\|_{\mathbf H}<\infty.$
In this case,
\begin{align}\label{bdd}
\|a\|_{\mathbf H}\leq \|\mathcal R_{a}\| \leq 2\sqrt{2}\|a\|_{\mathbf H}
\end{align}
\item The Rhaly operator $\mathcal R_{a}$ is compact on $H^2(\mathbb T)$ if and only if \[\lim_{N\to\infty}N\sum_{n=N}^{\infty}|a_n|^2=0.\]
\end{enumerate}
\end{proposition}
\begin{corollary}\label{E}
The bounded Rhaly operator $\mathcal R_{a}\in \mathbf E$ if and only if 
$$\lim_{N\to\infty} N\sum_{n=N}^{\infty}|a_{n+1}-a_n|^2=0.$$
\end{corollary}
\begin{proof}
By \cite[p.621]{BarriaHalmos1982}, we have
$\mathcal R_{a}\in \mathbf E$ if and only if $S^{*}\mathcal R_{a}S-\mathcal R_{a}$ is compact.
Moreover, \[S^{*}\mathcal R_{a}S-\mathcal R_{a}=\mathcal R_{S^*a-a},\]
where $S^*a-a=(a_2-a_1,a_3-a_2,\cdots,a_{n+1}-a_n,\cdots).$
By Proposition \ref{bddcompact}(2), the conclusion holds.
\end{proof}
\begin{theorem}
Every bounded Rhaly operator $\mathcal R_a$ belongs to $\mathbf T^{\infty}.$
\end{theorem}
\begin{proof}
For every positive integer $k,$
\begin{align*}
S^{*k}\mathcal R_aS^k=\mathcal R_{S^{*k}a},
\end{align*}
where $S^{*k}a=(a_{k+1},a_{k+2},\cdots,a_{k+n},\cdots).$
Let $p(z)=\sum_{i=1}^{N}b_iz^{i-1}$ be an analytic polynomial.
\begin{align*}
\|S^{*k}\mathcal R_aS^kp\|^2=\|\sum_{n=1}^{\infty}(a_{k+n}\sum_{i=1}^nb_i)z^{n-1}\|^2\leq \left(\sum_{i=1}^N |b_i|\right)^2\sum_{n=1}^{\infty}|a_{k+n}|^2\to 0, \quad k\to \infty.
\end{align*}
Since the set of all analytic polynomials is dense in $H^{2}(\mathbb D)$ and $\sup_k\|S^{*k}\mathcal R_aS^k\|\leq\|\mathcal R_a\|,$
\begin{align}\label{rasymbolzero}
S^{*k}\mathcal R_a S^k\xrightarrow{\mathrm{SOT}}0.
\end{align}
Therefore, $\mathcal R_a\in \mathbf T^{\infty}.$
\end{proof}
\begin{lemma}\label{lemma:p}
For every $p>1,N\geq1,$ we have 
\begin{align}
\sum_{n\geq N}\frac{1}{n^p}\leq \frac{p}{(p-1)N^{p-1}}.
\end{align}
\begin{proof}
Since $x^{-p}$ is decreasing on $(1,\infty)$,
$\frac{1}{(N+k+1)^{p}}\leq \int_{N+k}^{N+k+1}\frac{1}{x^p}dx(k\geq 0).$ Therefore,
\begin{align*}
\sum_{n\geq N}\frac{1}{n^p}\leq \frac{1}{N^p}+ 
\sum_{k=0}^{\infty}\int_{N+k}^{N+k+1}\frac{1}{x^p}dx
\leq \frac{1}{N^{p-1}}+\int_{N}^{\infty}\frac{1}{x^p}dx
=\frac{p}{(p-1)N^{p-1}}.
\end{align*}
\end{proof}
\end{lemma}
\begin{proposition}\label{property:hs}
\begin{enumerate}
  \item $\mathbf B\subset \mathbf H\subset \ell^2.$
  \item  Suppose that  $a=(a_n)_{n \ge 1},b=(b_n)_{n \ge 1}\in \mathbf H.$
$\forall \delta>0,$ we have 
\begin{enumerate}
\item $\omega_{a}(\delta)\leq  2\|a\|_{\mathbf H},$
\item $\omega_{a+b}(\delta)\leq  \omega_{a}(\delta)+\omega_{b}(\delta).$
\end{enumerate}
\item If $c=(c_n)_{n \ge 1}\in \mathbf B$ and $\delta >0,$ then 
$\omega_{c}(\delta)\leq \frac{2}{\sqrt{3}}B_{c}\delta.$
\end{enumerate}
\end{proposition}
\begin{proof}
(1) Let $a=(a_n)_{n \ge 1}\in \mathbf{B}.$ Lemma \ref{lemma:p} implies that 
\[N\sum_{n\geq N}|a_n|^2\leq N\sum_{n\geq N}\frac{B^2_a}{n^2}\leq 2B^2_a.\]
Therefore, $a\in \mathbf{H}.$ Let  $b=(b_n)_{n \ge 1}\in \mathbf{H}.$ Choose $N=1$ in the norm $\|b\|_{\mathbf{H}},$ we have $\sum_{n=1}^{\infty}|b_n|^2\leq \|b\|^2_{\mathbf{H}}.$ Thus, $b\in \ell^2.$

(2) For $\delta>0$ and $N\geq 1.$ Using Minkowski's inequality, we have 
\begin{align*}
\left(N\sum_{n\geq N}|a_{n+k}-a_n|^2\right) ^{1/2}&
\leq \left(N\sum_{n\geq N}|a_{n+k}|^2\right) ^{1/2}
+\left(N\sum_{n\geq N}|a_n|^2\right) ^{1/2}\\
&=\left(\frac{N}{k+N} (k+N)\sum_{n\geq k+N}|a_{n}|^2\right) ^{1/2}+\left(N\sum_{n\geq N}|a_n|^2\right) ^{1/2}\\
&\leq 2\|a\|_{\mathbf H}.
\end{align*}
Therefore, $\omega_{a}(\delta)\leq  2\|a\|_{\mathbf H}.$
Using Minkowski's inequality again, we have $\omega_{a+b}(\delta)\leq  \omega_{a}(\delta)+\omega_{b}(\delta).$

(3)For $N\geq 1$ and $0\leq k\leq \delta N,$ we have 
\begin{align*}
|c_{n+k}-c_n|\leq \sum_{i=0}^{k-1}|c_{n+i+1}-c_{n+i}|
\leq\sum_{i=0}^{k-1}\frac{B_c}{(n+i)^2}\leq \frac{kB_c}{n^2}\leq \frac{\delta N B_c}{n^2}.
\end{align*}
By Lemma \ref{lemma:p},
we have 
\begin{align*}
N\sum_{n\geq N}|c_{n+k}-c_n|^2\leq \delta^2 N^3 B^2_{c}\sum_{n\geq N}\frac{1}{n^4}\leq \frac{4}{3}\delta^2  B^2_{c}.
\end{align*}
Therefore, $\omega_{c}(\delta)\leq \frac{2}{\sqrt{3}}B_{c}\delta.$
\end{proof}
\begin{proposition}\label{hsmain1}
If $a=(a_n)_{n \ge 1}\in \operatorname{clos}_{\|\cdot\|_{\mathbf H}}\mathbf B,$ then $\lim\limits_{\delta\to 0}\omega_{a}(\delta)=0.$
\end{proposition}
\begin{proof}
Suppose that $a=(a_n)_{n \ge 1}\in \operatorname{clos}_{\|\cdot\|_{\mathbf H}}\mathbf B.$ $\forall \varepsilon>0,$ there exists $c=(c_n)_{n \ge 1}\in \mathbf B$ with constant $B_{c}(\varepsilon)$ such that 
$\|a-c\|_{\mathbf H}\leq \frac{\varepsilon}{2}.$
By Proposition \ref{property:hs}, we have 
\begin{align*}
\omega_{a}(\delta)\leq \omega_{a-c}(\delta)+\omega_{c}(\delta)
\leq 2\|a-c\|_{\mathbf H}+\frac{2}{\sqrt{3}}B_{c}(\varepsilon)\delta.
\end{align*}
Taking the limit as $\delta\to 0$ gives $\lim_{\delta\to 0}\omega_{a}(\delta)\leq \varepsilon.$ Therefore, $\lim_{\delta\to 0}\omega_{a}(\delta)=0.$
\end{proof}
\begin{lemma}\label{denseset}
A sequence $a\in \mathbf{B}$ if and only if there exists a bounded Lipschitz function $b$ on $\mathbb R$ such that
$a_n=\frac{b(\log n)}{n}\ (n\ge 1).$
\end{lemma}
\begin{proof}
Assume that there exists a bounded Lipschitz function $b$ on 
$\mathbb R$ such that
$a_n=\frac{b(\log n)}{n}$ and $|b(x)|\leq M.$ It follows that $|a_n|\leq \frac{M}{n}.$
\begin{align*}
|a_{n+1}-a_{n}|\leq \frac{|b(\log(n+1))-b(\log(n))|}{n+1}+\frac{|b(\log(n))|}{n(n+1)}.
\end{align*}
Let $L$ be the Lipschitz constant of $b.$ We have $|b(\log(n+1))-b(\log(n))|\leq L\log(1+1/n)\leq L/n$ and $\frac{|b(\log(n))|}{n(n+1)}\leq M/{n^2}.$ Therefore, $|a_{n+1}-a_{n}|\leq(L+M)/n^2 $ and $a\in \mathbf B.$

Conversely, suppose that $a\in \mathbf{B},$ then $n|a_n|\leq C$ and 
\begin{align}\label{y}
|(n+1)a_{n+1}-na_n|\leq|(n+1)(a_{n+1}-a_{n})|+|a_n|\leq \frac{(n+1)C}{n^2}+\frac{C}{n}\leq \frac{3C}{n}.
\end{align}
Since $\log(x+1)\geq \frac{x}{1+x}(x\geq 0),$
\begin{align}\label{x}
\log(n+1)-\log(n)=\log(1+\frac{1}{n})\geq \frac{1}{n+1}\geq\frac{1}{2n}.
\end{align}
Combining \eqref{y} and \eqref{x}, we have 
\begin{align}\label{Lipschitz}
\frac{|(n+1)a_{n+1}-na_n|}{\log(n+1)-\log(n)}\leq 6C.
\end{align}
Define the function $b$ as
$$b(x)=
\begin{cases} a_1, & x<\log1=0,\\ 
na_n+\frac{x-\log(n)}{\log(n+1)-\log(n)}((n+1)a_{n+1}-na_n), & \log(n)\leq x \leq \log(n+1). 
\end{cases}$$
It follows that $b$ is continuous and bounded by $C$ on $\mathbb R.$
By \eqref{Lipschitz}, $b$ is Lipschitz on each $[\log(n),\log(n+1)]$ with Lipschitz constant $6C.$ The triangle inequality implies that $b$ is  $6C$-Lipschitz continuous on $\mathbb R.$
\end{proof}
For $a=(a_n)_{n\ge 1}\in \mathbf{H},$ define the function $a—_{pc}:[1,\infty)\to\mathbb C$ as 
\begin{align*}
a_{pc}(x)=a_{\lfloor x\rfloor },\qquad x\in [1,\infty).
\end{align*}
Note that
\begin{align}\label{hata}
\int_{y}^{\infty}|a_{pc}(t)|^2dt\leq \sum_{n\geq \lfloor y\rfloor}|a_n|^2\leq \frac{\|a\|^2_{\mathbf H}}{\lfloor y\rfloor}\leq \frac{2\|a\|^2_{\mathbf H}}{y}\qquad y\geq 1.
\end{align}
\begin{proposition}\label{gamma0}
Let $a=(a_n)_{n \ge 1}\in \mathbf H.$
If \ $\lim_{\delta\to 0}\omega_{a}(\delta)=0,$ then $a\in \operatorname{clos}_{\|\cdot\|_{\mathbf H}}\mathbf B.$ 
\end{proposition}
\begin{proof}
Assume that $\lim_{\delta\to 0}\omega_{a}(\delta)=0.$ For $\varepsilon >0,$ there exists $\delta_0\in (0,1)$ such that $\omega_{a}(3\delta_0)<\varepsilon/3.$ 
Let 
\begin{align*}
\tilde{a}(x)=\frac{1}{\delta_0 x}\int_{x}^{x+\delta_0 x}a_{pc}(t)dt,\quad 
F(x)=\frac{1}{\delta_0}\int_{x}^{x+\delta_0 x}\tilde{a}(y)dy, \quad x\geq 1.
\end{align*}
Then
\begin{equation}\label{tildea}
\begin{aligned}
|x\tilde{a}(x)| &\leq \frac{1}{\delta_0}\int_{x}^{x+\delta_0 x}|a_{pc}(t)|dt\\
& \leq \frac{1}{\delta_0}\left(\int_{x}^{x+\delta_0 x}dt\right)^{1/2}\left(\int_{x}^{x+\delta_0 x}|a_{pc}(t)|^2dt\right)^{1/2}\\
& \leq \frac{1}{\delta_0}(\delta_0 x)^{1/2}\left(\int_{x}^{\infty}|a_{pc}(t)|^2dt\right)^{1/2}\\
& \leq \sqrt{\frac{2}{\delta_0}}\|a\|_{\mathbf H}. 
\end{aligned}
\end{equation}
The second inequality follows from the Cauchy-Schwarz inequality. The last inequality follows form \eqref{hata}.
It follows that
\begin{align}\label{Gbdd}
|F(x)|\leq \frac{1}{\delta_0}\int_{x}^{x+\delta_0 x}\sqrt{\frac{2}{\delta_0}}\|a\|_{\mathbf H}\frac{1}{t}dt\leq \sqrt{\frac{2}{\delta_0}}\|a\|_{\mathbf H}\frac{\log(1+\delta_0)}{\delta_0}\leq\sqrt{\frac{2}{\delta_0}}\|a\|_{\mathbf H}.
\end{align}
Since $\tilde{a}$ is continuous,
the fundamental theorem of calculus gives that $F$ is differentiable on $[1,\infty).$
Moreover, 
$xF^{'}(x)=\frac{1}{\delta_0}\left((x+\delta_0x)\tilde{a}(x+\delta_0x)-x\tilde{a}(x)\right).$ \eqref{tildea} implies that 
\begin{align}\label{daoshubdd}
|xF^{'}(x)|\leq 2\sqrt{2}\delta_0^{-3/2}\|a\|_{\mathbf H}.
\end{align}
Let
\[
b(t)=
\begin{cases}
F(1), & t<0,\\[2pt]
F(e^{t}), & t\ge 0,
\end{cases}
\qquad\text{and}\qquad
c_n=\frac{b(\log n)}{n},\quad n\ge 1.
\]
\eqref{Gbdd} implies that $b$ is bounded. Since $b'(t)=e^{t}F'(e^t)(t>0)$ \eqref{daoshubdd} and the mean value theorem imply that $b$ is Lipschitz. By Lemma \ref{denseset}, we have $c=(c_n)_{n\geq 1}\in \mathbf B.$
Applying Fubini's theorem, we have
\begin{equation}
\begin{aligned}
c_n=\frac{F(n)}{n}&=\frac{1}{\delta_0 n}\int_{n}^{n+\delta_0 n}
\left(\frac{1}{\delta_0 x}\int_{x}^{x+\delta_0 x}a_{pc}(t)dt\right)dx\\
&=\frac{1}{\delta^2_0n}\int_{n}^{(1+\delta_0) n}a_{pc}(t)\left(\int_{n}^{t}\frac{1}{x}dx\right)dt
+\frac{1}{\delta^2_0n}\int_{(1+\delta_0) n}^{(1+\delta_0)^2 n}a_{pc}(t)\left(\int_{\frac{t}{1+\delta_0}}^{(1+\delta_0) n}\frac{1}{x}dx\right)dt\\
&=\int_{1}^{\infty}K_n(t)a_{pc}(t)dt.
\end{aligned}
\end{equation}
Where 
\begin{align}\label{formula:kn}
K_n(t)=\frac1{\delta_0^2n} \begin{cases} \displaystyle \log\frac tn, &n\le t\le(1+\delta_0)n,\\[8pt] \displaystyle \log\frac{(1+\delta_0)^2n}{t}, &(1+\delta_0)n\le t\le(1+\delta_0)^2n,\\[8pt] 0,&\text{otherwise}. \end{cases}
\end{align}
Note that $\frac{1}{\delta_0 n}\int_{n}^{n+\delta_0 n}
\left(\frac{1}{\delta_0 x}\int_{x}^{x+\delta_0 x}a_{pc}(t)dt\right)dx=\int_{1}^{\infty}K_n(t)a_{pc}(t)dt.$
This formula holds  even if $a_{pc}$ is replaced with any locally integrable function. Therefore, when $a_{pc}$ is replaced with the constant function 1, we get 
\begin{align}\label{probabilitymeasure}
\int_{1}^{\infty}K_n(t)dt=1.
\end{align}
Applying the inequality $\log(1+x)\leq x \ (x>0)$ to \eqref{formula:kn}, we get
\begin{align}
K_{n}(t)\leq \frac{1}{\delta_0 n}. 
\end{align}
By Cauchy-Schwarz inequality with respect to the measure $K_n(t)dt$ and \eqref{probabilitymeasure}, we have 
$|a_n-c_n|^2\leq\int_{1}^{\infty}|a_n-a_{pc}(t)|^2 K_n(t)dt
\leq \frac{1}{\delta_0 n}\int_{n}^{(1+\delta_0)^2n}|a_n-a_{\lfloor t \rfloor}|^2dt.$
Note that $n\leq \lfloor t \rfloor\leq (1+\delta_0)^2n.$ Let $\lfloor t \rfloor=n+i.$ Since $\delta_0\in (0,1),$ we have $0\leq i\leq (2\delta_0+\delta^2_0)n\leq 3\delta_0 n.$ Hence
\begin{align}
|a_n-c_n|^2\leq \frac{1}{\delta_0 n}\sum_{i=0}^{\lfloor 3\delta_0 n\rfloor}|a_{n+i}-a_n|^2
=\frac{1}{\delta_0 n}\sum_{i=1}^{\lfloor 3\delta_0 n\rfloor}|a_{n+i}-a_n|^2.
\end{align}
Let $I_i=\max\{N,\lceil i/{3\delta_0}\rceil\}.$
For $N\geq 1,$ we have 
\begin{align*}
N\sum_{n\geq N}|a_n -c_n|^2\leq \frac{N}{\delta_0}\sum_{n\geq N}\frac{1}{n}\sum_{i=1}^{[3\delta_0 n]}|a_{n+i}-a_n|^2=\frac{N}{\delta_0}\sum_{i\geq 1}\sum_{n\geq I_i}\frac{|a_{n+i}-a_n|^2}{n}.
\end{align*}
If $n\geq N$, then $i\leq 3\delta_0 n$
if and only if $n\geq I_i.$ Since 
$\lceil i/{3\delta_0}\rceil \geq i/{3\delta_0},$ 
\begin{align}\label{iIi}
i\leq 3\delta_0 \lceil i/{3\delta_0}\rceil \leq 3\delta_0 I_i.
\end{align}
Moreover, $\sum_{n\geq I_i }|a_{n+i}-a_n|^2\leq \omega_{a}^2(3\delta_0)/I_{i}.$ Since $n\geq I_i,$
\begin{align*}
N\sum_{n\geq N}|a_n -c_n|^2\leq \frac{N\omega_{a}^2(3\delta_0)}{\delta_0}\sum_{i\geq 1}\frac{1}{I^2_i}
=\frac{N\omega_{a}^2(3\delta_0)}{\delta_0}(\sum_{1\leq i\leq 3\delta_0 N}\frac{1}{I^2_i}+\sum_{i>3\delta_0 N}\frac{1}{I^2_i}).
\end{align*}
By $I_i\geq N,$
\begin{align*}
\sum_{1\leq i\leq 3\delta_0 N}\frac{1}{I^2_i}\leq \sum_{1\leq i\leq 3\delta_0 N}\frac{1}{N^2}\leq\frac{3\delta_0}{N}.
\end{align*}
By \eqref{iIi}, we get $1/I^2_i\leq 9\delta^2_0/i^2.$ Using Lemma \ref{lemma:p}, we have
\begin{align}
\sum_{i>3\delta_0 N}\frac{1}{I^2_i}
\leq 9\delta^2_0\sum_{i>3\delta_0 N}\frac{1}{i^2}=9\delta^2_0\sum_{i\geq \lfloor 3\delta_0 N\rfloor+1}\frac{1}{i^2}
\leq 9\delta^2_0 \frac{2}{\lfloor 3\delta_0 N\rfloor+1}\leq 9\delta^2_0 \frac{2}{3\delta_0 N}\leq \frac{6\delta_0}{N}.
\end{align}
It follows that $N\sum_{n\geq N}|a_n -c_n|^2\leq \frac{N\omega_{a}^2(3\delta_0)}{\delta_0}(\frac{3\delta_0}{N}+\frac{6\delta_0}{N})\leq 9\omega_{a}^2(3\delta_0)$
and \[\|a-c\|_{\mathbf H}\leq 3\omega_{a}(3\delta_0)< \varepsilon.\]
Therefore, $a\in \operatorname{clos}_{\|\cdot\|_{\mathbf H}}\mathbf B.$ 
\end{proof}
\begin{proposition}\label{BQ}
If $a=(a_n)_{n \ge 1}\in \mathbf B,$ then $\mathcal R_{a}\in \mathbf Q.$
\end{proposition}
\begin{proof}
Assume that $a=(a_n)_{n \ge 1}\in \mathbf B,$ by Lemma \ref{denseset}, we have $a_n=\frac{b(\log n)}{n}$ for some bounded Lipschitz function $b$ 
with Lipschitz constant $L.$
By \eqref{eq:cinftymatrix}, we have 
\begin{align}\label{formula:cbR}
E^{*}\mathcal C_b E=\mathcal R_c+diag(\gamma_n)_{n\geq 1},
\end{align}
where 
$c=\{d_1,c_2,c_3,\cdots,c_n,\cdots\}$ and 
$\gamma_1=0,\gamma_n=d_n-c_n=\int_{n-1}^{n}b(\log x)\frac{x-n}{x}dx,\ n\geq 2.$ Note that
\begin{align*}
|\gamma_n|\leq \frac{\|b\|_{\infty}}{n-1}\int_{n-1}^{n}(n-x)dx=\frac{\|b\|_{\infty}}{2(n-1)}.
\end{align*}
So $\sum_{n\geq 2}|\gamma_n|^2<\infty$ and $diag(\gamma_n)_{n\geq 1}$ is Hilbert-Schmidt.

For $n\geq 2$ and $n-1<x<n,$ we have 
\begin{align*}
\left|\frac{b(\log x)}{x}-\frac{b(\log n)}{n}\right|
&\leq \left|b(\log x)(\frac{1}{x}-\frac{1}{n})\right|
+\left|\frac{1}{n}(b(\log x)-b(\log n))\right|\\
&\leq \frac{\|b\|_{\infty}(n-x)}{nx}+\frac{L}{n}\log (1+\frac{n-x}{x})\\
&\leq \frac{(\|b\|_{\infty}+L)(n-x)}{nx}.
\end{align*}
Integrating both sides of the above inequality yields
\begin{align*}
|a_n-c_n|&\leq \int_{n-1}^{n}\frac{(\|b\|_{\infty}+L)(n-x)}{nx}dx\\
&\leq \frac{\|b\|_{\infty}+L}{n(n-1)}\int_{n-1}^{n}(n-x)dx
=\frac{\|b\|_{\infty}+L}{2n(n-1)}
\end{align*}
and 
$\sum_{n=2}^{\infty}n|a_n-c_n|^2<\infty.$
It follows that $\mathcal R_a - \mathcal R_{c}$ is  Hilbert-Schmidt.
Note that 
$\mathcal R_a=E^*\mathcal C_bE+(\mathcal R_a-\mathcal R_c)
 -diag(\gamma_n)_{n\geq 1}.
$
Since $\mathbf Q$ contains the ideal of compact operators,  Theorem \ref{the:cb} implies that $\mathcal R_a\in \mathbf Q.$
\end{proof}
\subsection{Proof of Theorem \ref{thm:main2}}
\begin{proof}[Proof of Theorem \ref{thm:main2}]
Assume that $\mathcal R_a\in\mathbf T.$
Since $\mathbf T\subset \mathbf T^{+},$ Corollary \ref{T+} implies that 
$\lim\limits_{\delta\to 0}
\sup_{\substack{N\geq1\\0\leq k\leq \delta N}}
\left(N\sum_{n=N}^\infty|a_{n+k}-a_n|^2\right)=0.$

Suppose that the sequence $a=(a_n)_{n\geq 1}$ satisfies
$\lim\limits_{\delta\to 0}
\sup_{\substack{N\geq1\\0\leq k\leq \delta N}}
\left(N\sum_{n=N}^\infty|a_{n+k}-a_n|^2\right)=0.$
By Proposition \ref{gamma0} and Lemma \ref{denseset}, for $\varepsilon>0,$ there exists a bounded Lipschitz funciton $b$ such that $\|a-d\|_{\mathbf H}<\varepsilon,$ where $d=(\frac{b(\log n)}{n})_{n\geq1}.$ 
Using \eqref{bdd}, we have 
\begin{align}
\|\mathcal R_a -\mathcal R_d\|\leq 2\sqrt{2}\|a-d\|_{\mathbf H}. 
\end{align}
Proposition \ref{BQ} implies $\mathcal R_d\in \mathbf Q.$
Hence, $\mathcal R_a\in \mathbf Q\subset \mathbf T.$

If $\mathcal R_a\in \mathbf T_c,$ then \eqref{rasymbolzero} implies the symbol of $\mathcal R_a$ is zero. Since the commutator ideal of $\mathbf T_c$ is the ideal of compact operators, we have $\mathcal R_a$ is compact.
Conversely, if $\mathcal R_a$ is compact, then $\mathcal R_a$ is in the commutator ideal of $\mathbf T_c.$ 
Using Proposition \ref{bddcompact}(2), we obtain the equivalence
of conditions (a), (b), and (c) in Theorem \ref{thm:main2}.
\end{proof}
\subsection{Proof of Theorem \ref{thm:main3}}
Let $0<\gamma<1,b_n=e^{in^{\gamma}}/n$ and $b=\{b_n\}_{n\geq 1}.$ 
It follows that $\mathcal R_b=diag(e^{in^{\gamma}})_{n\geq 1}\mathcal C$ and $\|\mathcal R_b\|\leq \|\mathcal C\|.$
Note that
\begin{align}\label{cha}
b_{n}-b_{n+k}=\frac{1}{n}(e^{in^{\gamma}}-e^{i(n+k)^{\gamma}})+\frac{k}{n(n+k)}e^{i(n+k)^{\gamma}}
\end{align}
Since $|e^{ix}-e^{iy}|=2|\sin(\frac{x-y}{2})|\leq |x-y|$,
the mean value theorem implies that
\begin{align*}
|b_{n+1}-b_n|\leq \frac{1}{n}|n^{\gamma}-(n+1)^{\gamma}|+\frac{1}{n(n+1)}
\leq \gamma n^{\gamma-2}+n^{-2}
\leq (1+\gamma)n^{\gamma-2}.
\end{align*}
By Lemma \ref{lemma:p}, we have
\begin{align*}
N\sum_{n\geq N}|b_{n+1}-b_n|^2
\leq N(1+\gamma)^2\sum_{n\geq N}n^{2\gamma-4}
\leq (1+\gamma)^2\frac{4-2\gamma}{(3-2\gamma)N^{2-2\gamma}}\to 0\qquad N\to\infty.
\end{align*}
By Corollary \ref{E}, we have $\mathcal{R}_{b}\in \mathbf E.$

Next we will prove $\mathcal R_b\notin \mathbf T^{+}.$
For every $\delta>0,$ there exists $N_0$ such that $N^{1-\gamma}\leq \delta N (N\geq N_0).$ Take $k_{N}=[N^{1-\gamma}].$ Thus $0\leq k_N\leq \delta N.$ By \eqref{cha}, we have 
\begin{align*}
|b_{n+k_{N}}-b_n|\geq 
\frac{1}{n}\left(\left|e^{i(n+k_N)^{\gamma}-in^{\gamma}}-1\right|-\frac{k_N}{n+k_N}\right).
\end{align*}
Suppose that $N\leq n\leq 2N.$ Since $x^{\gamma-1}$ is decreasing,
$$\gamma k_N(n+k_N)^{\gamma-1} \leq(n+k_N)^\gamma-n^\gamma \leq\gamma k_Nn^{\gamma-1}.$$
The inequalities $N\leq n$ and $k_N\leq N^{1-\gamma}$ imply that 
$\gamma k_Nn^{\gamma-1}\leq \gamma.$
The inequalities  $k_N\leq N$ and $n+k_N\leq 3N$ imply that 
$\gamma k_N(n+k_N)^{\gamma-1}\geq \gamma k_N(3N)^{\gamma-1}.$
There exists $N_1\geq N_0$ such that 
$k_N N^{\gamma-1}\geq \frac{1}{2}.$ Therefore,
$\frac{\gamma3^{\gamma-1}}{2}\leq(n+k_N)^\gamma-n^\gamma \leq\gamma$
and
\begin{align}\label{sin}
\left|e^{i(n+k_N)^{\gamma}-in^{\gamma}}-1\right|=2\sin\frac{|(n+k_N)^{\gamma}-n^{\gamma}|}{2}>2\sin\frac{\gamma}{3^{1-\gamma} 4}.
\end{align}
Since $\frac{k_N}{n+k_N}\le\frac{k_N}{N}\le N^{-\gamma} \longrightarrow0$ as $N\to\infty,$ there exists 
$N_2\geq N_1$ such that 
\begin{align}\label{kn}
\frac{k_N}{n+k_N}\leq \sin\frac{\gamma}{3^{1-\gamma} 4} \text{for all} \ N\geq N_2.
\end{align}
Combining \eqref{sin} with \eqref{kn}, we get $|b_{n+k_{N}}-b_n|\geq \frac{1}{n}\sin\frac{\gamma}{3^{1-\gamma} 4}.$
\begin{align*}
N \sum_{n\geq N}|b_{n+k_N}-b_n|^2&\geq N\sum_{n=N}^{2N}|b_{n+k_N}-b_n|^2\\
&\geq N \sin^2\frac{\gamma}{3^{1-\gamma} 4}\sum_{n=N}^{2N}\frac{1}{n^2}\\
&\geq N \sin^2\frac{\gamma}{3^{1-\gamma} 4}\frac{N+1}{4N^2}\\
&\geq \frac{1}{4}\sin^2\frac{\gamma}{3^{1-\gamma} 4}(1+\frac{1}{N})
\geq \frac{1}{4}\sin^2\frac{\gamma}{3^{1-\gamma} 4}>0.
\end{align*}
By Corollary \ref{T+}, we have $\mathcal R_b\notin \mathbf T^{+}.$

\section{AI Disclosure.}
The proof of Theorem~\ref{key2} and some of the estimates
were obtained with the assistance of AI tools.
All arguments and estimates were checked by the author,
who takes full responsibility for the mathematical content.

\end{document}